\documentclass[preprint,12pt]{elsarticle}
\usepackage[margin=2.5 cm]{geometry}
\usepackage{amssymb,amsmath,amsthm,epsfig,float}
\usepackage{amsfonts}

\usepackage{graphics,psfrag,graphicx,color}
\usepackage{epstopdf}
\usepackage{caption}
\usepackage{subfigure}
\usepackage{pgfplots}
\usepackage{fixmath}
\usepackage{tikz}
\usetikzlibrary{intersections,shapes.arrows,calc,automata}
\usetikzlibrary{arrows,decorations.markings}
\usetikzlibrary{plotmarks}

\DeclareMathOperator{\sech}{sech} 
\usepackage[ruled,vlined]{algorithm2e}
\usepackage{listings}

\usepackage{xcolor}
\usepackage{multirow}
\definecolor{dkgreen}{rgb}{0,0.6,0}
\definecolor{gray}{rgb}{0.5,0.5,0.5}
\definecolor{mauve}{rgb}{0.58,0,0.82}

\newcommand{\esref}[1]{Equation~(\ref{#1})}

\newcommand{\fref}[1]{Figure~\ref{#1}}

\newtheorem{assumption}{Assumption}
\newtheorem{lemma}{Lemma}
\newtheorem{remark}{Remark}[section]
\newtheorem{theorem}{Theorem}[section]
\newcommand{\answ}[1]{{\color{black}{#1}}}
\newcommand{\reviewer}[1]{{\textcolor{blue}{#1}}}
\usepackage[normalem]{ulem}

\begin{document}
\begin{frontmatter}
\title{Virtual element method for semilinear sine-Gordon equation over polygonal mesh using product approximation technique}
\author[ind1]{D Adak}
\author[ind1]{S Natarajan\corref{cor1}\fnref{fnlab1}}

\address[ind1]{Integrated Modelling and Simulation Lab, Department of Mechanical Engineering, Indian Institute of Technology Madras, Chennai-600036, India.}

\cortext[cor1]{Corresponding author}
\fntext[fnlab1]{Department of Mechanical Engineering, Indian Institute of Technology Madras, Chennai-600036, India. E-mail:sundararajan.natarajan@gmail.com; snatarajan@iitm.ac.in}

\begin{abstract}
In this paper, we employ the linear virtual element spaces to discretize the semilinear sine-Gordon equation in two dimensions. The salient features of the virtual element method (VEM) are: (a) it does not require explicit form of the shape functions to construct the nonlinear and the bilinear terms, and (b) relaxes the constraint on the mesh topology by allowing the domain to be discretized with general polygons consisting of both convex and concave elements, \answ{and (c) easy mesh refinements (hanging nodes and interfaces are allowed )}. The nonlinear source term is discretized by employing the product approximation technique and for temporal discretization, the Crank-Nicolson scheme is used. The resulting nonlinear equations are solved using the Newton's method. We derive \answ{\textit{a priori}} error estimations in $L^2$ and $H^1$ norms. The convergence properties and the accuracy of the virtual element method for the solution of the sine-Gordon equation are demonstrated with  academic numerical experiments.


\end{abstract}

\end{frontmatter}
\section{Introduction}

The sine-Gordon equation is a nonlinear hyperbolic equation that finds application in a variety of problems in Science and Engineering, viz., shallow-water waves, optical fibers, Josephson-junction, mechanical transmission line, to name a few. The solution of the sine-Gordon equation consists of soliton and multi-soliton solutions. In this direction, several numerical techniques have been employed to study this problem, that includes, finite difference method~\cite{ben1986numerical}, the finite element method~\cite{argyris1991finite}, the meshless methods~\cite{dehghan2010numerical} and the boundary element method~\cite{dehghan2008dual}. Ben-Yu  et al. \cite{ben1986numerical} 
employed the finite difference scheme for the two-dimensional undamped sine-Gordon equation. By exploiting the product approximation techniques, Argyris  et al. \cite{argyris1991finite} studied the 2D sine-Gordon equation in the context of the finite element method. Further, Christiansen and Lomdahl \cite{christiansen1981numerical} employed the generalized leapfrog method to study the solution of the two-dimensional sine-Gordon equation. Aforementioned studies focused on undamped sine-Gordon equation. The damped sine-Gordon equation demands more sophisticated numerical techniques. In this direction, the following articles are presented in the literature: Nakajima  et al. \cite{nakajima1974numerical}  proposed an efficient numerical technique for the solution of the sine-Gordon equation considering dimensionless loss factors and unit less normalized bias. Recently, a dual reciprocity boundary element method~\cite{dehghan2008dual} and the radial basis function method~\cite{dehghan2008numerical} have been introduced for the above mentioned model problem. However, most of these techniques require the domain to be discretized with simplex elements such as triangles/quadrilaterals.

The introduction of elements with arbitrary number of sides has revolutionized the simulation technique. Elements with arbitrary number of sides offer greater flexibility in meshing and \answ{are} less sensitive to mesh distortion~\cite{sukumarmalsch2006,manzinirusso2014}. Among the different approaches based on polygonal elements, such as the scaled boundary finite element method~\cite{natarajanooi2014}, the polygonal FEM~\cite{sukumarmalsch2006,szesheng2005,bishop2014}, the strain smoothing techniques~\cite{liutrung2010,natarajanbordas2015,francisa.ortiz-bernardin2017}, the virtual element method (VEM)~\cite{beirao2013basic,beirao2016virtual,vacca2015virtual} has an edge. The VEM is a recently developed framework, which provides \answ{FEM-like} approximations over arbitrary polygonal meshes. Some of the salient features of the VEM are: (a) it does not require an explicit form of the shape functions to compute the  bilinear form; (b) it can be applied to convex, and concave elements, i.e., star convexity is not required; (c)\answ{ easy mesh refinements (hanging nodes and interfaces are allowed )}  and (d) the degrees of freedom associated with the VEM space, provides sufficient information to compute the discrete bilinear form. Due to its versatility, the method has been applied to a wide variety of problems. Some of them include: the Stokes equation \cite{antonietti2014stream}, the Plate bending equation \cite{brezzi2013virtual}, linear elliptic, parabolic, and hyperbolic equations \cite{beirao2013basic,beirao2016virtual,vacca2015virtual,vacca2016virtual}, linear and nonlinear elasticity problems~\cite{da2013virtual,da2015virtual}, the convection diffusion equation with small diffusion \cite{benedetto2016order}, eigenvalue problems \cite{mora2015virtual,certik2018virtual}, nonconforming VEM for Stokes and elliptic equations~\cite{cangiani2016nonconforming,de2016nonconforming,cangiani2016conforming}. Recently an open source C++ library has been developed by Ortiz-Bernardin et al.,~\cite{BernardinAlvarez2017}. Sutton~\cite{Sutton2017} developed a 50-line MATLAB implementation for the lowest order VEM for {the} two-dimensional Poisson's problem. Dhanush and Natarajan implemented the VEM into the commercial finite element software  Abaqus~\cite{dhanushnatarajan2018} using user defined subroutines for thermo-elasticity.


However, most of the highlighted works deal with linear model problems. To the best of the authors' knowledge, limited research has been focused on employing the VEM for nonlinear problems. By employing the standard linearization technique, Antonietti  et al.~\cite{antonietti2016c} introduced the $\mathcal{C}^1$ VEM for the Cahn-Hilliard problem.  Beir\~{a}o  da Veiga  {et} al.  designed a new projection operator to approximate the tri-linear term in the Navier-Stokes equation \cite{da2018virtual}. Recently Adak et al.~\cite{adak2019convergence,adak2018virtual}, extended the VEM to semilinear parabolic and hyperbolic equations. The nonlinear term was approximated by employing the $L^2$ projection operator. However, the difficulty \answ{with this approach} is that the computation of the Jacobian is numerically expensive, as the Jacobian has to be updated at each time step. To improve the robustness of the VEM, in this paper, we use the product approximation technique~\cite{argyris1991finite,christie1981product} to discretize the nonlinear source term in the sine-Gordon equation. The salient feature of the proposed scheme is that it is easy to implement, and it is computable, thus avoiding complicated Jacobian matrix formation. 

The rest of the manuscript is organized as follows. In Section~\ref{notation:sg}, we define mesh regularity {assumptions} for theoretical estimations and the basic notation of Sobolev spaces. In the same \answ{section}, we introduce the model problem and {the} weak formulation. In Section~\ref{virtual_sg} we review the fundamental setting  of the virtual element {spaces} and the corresponding semidiscrete and fully-discrete formulations. In the same section,  we define the product approximation techniques for the nonlinear source function in view of the virtual element setting. Error estimations for semi-discrete and fully discrete schemes in respective norms are derived in Section~\ref{converge:sg}. Finally, in Section \ref{num_ex_sg} we investigate the numerical solutions of the sine-Gordon equation with different boundary conditions. A comparison of the proposed approach with the existing technique \cite{adak2019convergence} is discussed in the same section, followed by concluding remarks in the last section.


\section{Governing equation and the weak form}
\label{notation:sg}
Consider the sine-Gordon equation in two space variables
\begin{equation}
    D^2_t u+ \gamma \  D_t u-\Delta u= f(u)
\label{eqn:mod_prob}
\end{equation}
in the \answ{convex polygonal region $\Omega \subset \mathbb{R}^2$ }and $t \in [0,T]$, where $\gamma$ is the dissipative coefficient and $T$ denotes the final time. The above equation is supplemented with the following boundary conditions:
\begin{equation}
    \begin{aligned}
      u(x,y,t)=0 \quad \answ{ \forall \ (x,y) \in \partial \Omega} ; \ 0 \leq t \leq T,  
    \end{aligned}
\end{equation}
and $f(u):=-\sin u$. Further, the initial conditions are: 
\begin{equation*}
\begin{aligned}
    u(x,y,0)=h(x,y), \quad \answ{\forall \ (x,y) \in  \Omega}, \\
    D_t u(x,y,0)=g(x,y), \quad  \answ{\forall \ (x,y) \in  \Omega}. 
    \end{aligned}
\end{equation*}
The functions $h(x,y)$ and $g(x,y)$ denote the wave modes and their velocity, respectively. Now, multiplying \esref{eqn:mod_prob} by the test function $v \in H^1_0(\Omega)$ and using Green's theorem, the weak formulation is given by:
Given $h(x,y),\ g(x,y) \in H^1_0(\Omega)$, find $u \in \mathcal{C}^0(0,T;H^1_0(\Omega))\cap \mathcal{C}^1(0,T;L^2(\Omega))$ such that~\cite{grisvard1992singularities}:
\begin{equation}
\label{cont_weak:SG}
    \left \{
    \begin{aligned}
    & (D^2_t u,v)+ \gamma \  (D_t u,v)+a(u,v)=(f(u),v) \  \forall v \in H^1_0(\Omega), \ \text{for a.e.} ~~ t \in (0,T], \\
    & u(0)=h(x,y), \\
    & D_t u(0)=g(x,y),
    \end{aligned}
    \right.
\end{equation}
where 
\begin{itemize}
    \item $D^2_t u \in L^2(0,T;L^2(\Omega))$ and $D_t u \in L^2(0,T;L^2(\Omega))$ denote the \answ{derivatives} of $u$ with respect to $t$.
    \item $(\cdot,\cdot)$ denotes the global $L^2(\Omega)$ inner product.
    \item $a(\cdot,\cdot)$ denotes the $H^1_0(\Omega)$ inner-product.
\end{itemize}
\esref{cont_weak:SG} is a \answ{second-order} nonlinear differential equation with respect to time $t$, where $f(u)$ is a global Lipschitz continuous function in $u$ and the bilinear form $a(\cdot,\cdot)$ is continuous and coercive, i.e.
\begin{equation*}
    a(u,v) \leq~C~\|u\|_{1,\Omega} \ \|v\|_{1,\Omega} \quad \text{and} \quad \answ{a(u,v) \geq \alpha ~|v|_{1,\Omega}} \quad \forall\  u,v \in H^1_0(\Omega).
\end{equation*}
Hence, problem \eqref{cont_weak:SG} has a unique solution~\cite{baker1980multistep,lions2012non}.

\section{The virtual element framework} 
\label{virtual_sg}
The virtual element method can be seen as a generalization of the finite element method (FEM) that relaxes the constraint on the mesh topology that the FEM otherwise imposes. In this section, we present a brief overview of the VEM.

\paragraph{Notations} Let $ \{ \mathcal{T}_{h} \}_{h} $ be a family of decompositions of $\Omega $ into polygonal elements. Let $K$ be an element with diameter $h_{K}$ and $N_K$ be the total number of vertices of a polygon $K$. Let us define mesh diameter $h:= \underset {K\in \mathcal{T}_{h}}{\text{max}} (h_{K}) $. For a Banach space $H$ with norm $\| \cdot \|_H$, and a function $\omega: [0,T] \rightarrow  H$, that is Lebesgue measurable, the following norms are defined:
\begin{equation*}
    \|\omega\|_{L^2(0,T;H)}:=\Big(\int_0^T \|\omega \|_{H}^2~ds \Big)^{1/2} \quad \text{and} \quad
\|\omega\|_{L^{\infty}(0,T;H)}:=\underset{0 \leq t \leq T}{\text{ess sup}}~\|\omega(\cdot,t)\|_{H}.
\end{equation*}
We define
\begin{equation*}
    L^{2}(0,T;H)=\{ \omega:(0,T] \rightarrow H:\|\omega\|_{L^{2}(0,T;H)} < \infty\}.
\end{equation*}
The local interpolation approximation properties and the stability of the discrete bilinear forms depend on the mesh regularity, which is stated below~\cite{beirao2013basic,ahmad2013equivalent}. These are later required for estimating the \answ{{\it a priori}} estimates in the respective norms ($L^2$ and $H^1$ norm).
\begin{assumption}
\label{assum_1}
\text{(Mesh regularity)}
\item (1) There exists a real number $\rho>0$ such that, for all $h$, each element $K \in \mathcal{T}_{h} $ is star-shaped with respect to a ball of radius $\geq \rho h_{K}$.
\item (2) There exists a constant \answ{$c \in (0,1)$} such that, for all $h$ and for each element $K\in \mathcal{T}_{h}$, the distance between any two vertices of $K$ is $\geq c\ h_{K}$.
\end{assumption}

The discrete bilinear forms for the model problem (\answ{cf.} \eqref{eqn:mod_prob}) are constructed with the help of the following two projection operators (defined element wise), 
\begin{itemize}
    \item The elliptic projection operator $\mathcal{P}^{\nabla}_{K}:H^{1}(K)\rightarrow \mathbb{P}_{1}(K)$ is defined by:
    \begin{equation}
    \label{ell_pro}
    \int_{K} \nabla(\mathcal{P}^{\nabla}_{K} u-u)\cdot \nabla p_1=0 \quad \forall \ p_1 \in \mathbb{P}_{1}(K) \quad \text{and} \quad P^0(\mathcal{P}^{\nabla}_{K} u-u)=0,
    \end{equation} 
    $\forall \  u \in H^1(K)$, where $P^{0}u$ is defined as \answ{ $P^{0}u :=\frac{1}{|\partial K|}~\int_{\partial K } u$}.
    \item The local $L^2$ projection operator $\mathcal{P}^{0}_{K}:L^2(K)\rightarrow \mathbb{P}_{1}(K)$, defined by: \cite{de2016nonconforming}
    \begin{equation}
    \label{l2_pro_p1}
    \Big(\mathcal{P}^{0}_{K}q-q,p_{1} \Big)_{0,K}=0 \quad \forall \ p_1 \in \mathbb{P}_{1}(K), ~\forall \  q \in L^2(K).
    \end{equation}
\end{itemize}
Globally, the projection operator is constructed as $(\mathcal{P}^{0} q)|_{K}=\mathcal{P}^{0}_{K}(q)$ for all $q \in L^2(\Omega)$.
With these auxiliary space is given by:
 \begin{equation*}
\mathcal{G}(K):=\Big\{ v \in H^{1}(K)\cap C^{0}(\partial K): v|_{e} \in \mathbb{P}_{1}(e) \ \forall e \in \partial K,\  \Delta v \in \mathbb{P}_{1}(K)\Big \}.
\end{equation*}
\answ{We would like to highlight that the elliptic projection operator $\mathcal{P}^{\nabla}_K$ defined in \eqref{ell_pro} is computable from the vertex values of a function which is in $H^1(K) \cap C^0(\bar{K})$. Hence it is compuatable for all the function that is in $\mathcal{G}(K)$. Further, using the projection operator $\mathcal{P}^{\nabla}_{K}$, we define the local modified VEM space \cite{ahmad2013equivalent} }
\begin{equation*}
\mathcal{Q}(K):= \Big\{ v \in \mathcal{G}(K):\int_{K} (\mathcal{P}^{\nabla}_{K}v) q=\int_{K} v\  q ~ \forall q \in \mathbb{P}_{1}(K)  \Big\} \quad \answ{~ K \in \mathcal{T}_{h}},
\end{equation*}
and the global virtual element space as
\begin{equation*}
\mathcal{Q}_{h}:= \Big \{ v \in H^{1}_{0}(\Omega):v|_{K}\in \mathcal{Q}(K)  \ \forall K \in \mathcal{T}_{h} \Big\}.
\end{equation*}
 In addition, we would like to highlight that in the modified virtual element space $\mathcal{Q}(K)$, we can compute the projection operator $\mathcal{P}^{0}_{K}$. The dimension of the local virtual element space $\mathcal{Q}(K)$ is the same as the linear space defined in \cite{beirao2013basic}, where the unisolvency is also proved. Further, we define the degrees of freedom (\answ{DOFs}) associated with $\mathcal{Q}(K)$:
\begin{itemize}
\item $({\bf d_1 })$ values of $v$ at the $N_K$ vertexes of $K$.
\label{ch_1_dof}
\end{itemize}
The polynomial space $\mathbb{P}_{1}(K)$ is subset of $\mathcal{Q}(K)$. However, for implementation purpose, we use the set of scaled monomials $\mathcal{M}(K)$  defined in the following manner:
\begin{equation*}
\mathcal{M}(K) = \Big\{ 1, \frac{x-x_K}{h_K},  \frac{y-y_K}{h_K} \Big\}.
\end{equation*}

\subsection{Discrete virtual element formulation}
Unlike the FEM, the VEM does not require an explicit form of the basis functions $\eta$ to compute the bilinear forms. By employing the elliptic projection operator $\mathcal{P}^{\nabla}_{K}$, we decompose $v_{h}\in \mathcal{Q}(K)$ as $v_{h}=\mathcal{P}^{\nabla}_{K}(v_{h})+(I-\mathcal{P}^{\nabla}_{K})(v_{h})$, where $\mathcal{P}^{\nabla}_{K}(v_{h})$ is the polynomial part and $(I-\mathcal{P}^{\nabla}_{K})(v_{h})$ is the nonpolynomial part. The polynomial part can be directly computed using the \answ{DOFs}, whilst the nonpolynomial part can only be approximated using the \answ{DOFs}. For each polygon $K\in \mathcal{T}_h$, the discrete bilinear form \mbox{$a^{K}_h(\cdot,\cdot):\mathcal{Q}(K) \times \mathcal{Q}(K)\rightarrow \mathbb{R}$} is defined as follows
\begin{equation}
\begin{split}
a^{K}_{h}(u_{h},v_{h})&:=a^{K}\Big(\mathcal{P}^{\nabla}_{K}(u_{h}),\mathcal{P}^{\nabla}_{K}(v_{h})\Big)+S^{K}_{a} \Big((I-\mathcal{P}^{\nabla}_{K})u_{h},(I-\mathcal{P}^{\nabla}_{K})v_{h}\Big),
\label{temp4}
\end{split}
\end{equation}
where $u_{h},v_{h} \in \mathcal{Q}(K)$ and $S^{K}_{a}(\cdot,\cdot)$ is a symmetric bilinear form that ensures stability of the discrete bilinear form $ a^{K}_{h}(\cdot,\cdot)$. Moreover, the stabilizer $S^{K}_{a}(\cdot,\cdot)$ \answ{must be} spectrally equivalent to the identity matrix and scales as $a^{K}(\cdot,\cdot)$.
The global bilinear form is then given by: \mbox{$a_{h}(\cdot,\cdot):\mathcal{Q}_{h} \times \mathcal{Q}_{h} \rightarrow \mathbb{R}$} by adding local contributions as
\begin{eqnarray}
a_{h}(u_{h},v_{h}):=\sum_{K \in \mathcal{T}_{h}} a^{K}_{h}(u_{h},v_{h}) \ \forall \ u_{h},v_{h} \in \mathcal{Q}_{h}.
\label{temp6}
\end{eqnarray}

Following \cite{vacca2015virtual,beirao2016virtual} and using the orthogonal $L^2$ projection operator $\mathcal{P}^{0}_{K}$, we can formulate the local bilinear form $m^{K}_{h}(\cdot,\cdot):\mathcal{Q}(K) \times \mathcal{Q}(K) \rightarrow \mathbb{R}$ for each polygon $K$  as
\begin{equation}
m^{K}_{h}(u_{h},v_{h}):= \Big(\mathcal{P}^{0}_{K}u_{h},\mathcal{P}^{0}_{K}v_{h}\Big)+S^{K}_{m}\Big((I-\mathcal{P}^{0}_{K})u_{h},(I-\mathcal{P}^{0}_{K})v_{h} \Big).   
\label{mass:SG}
\end{equation}
With the global form given by:
\begin{equation}
m_{h}(u_h,v_h) = \sum_{K \in \mathcal{T}_{h}} m^{K}_{h}(u_h,v_h) \quad \forall \  u_h,v_h \in \mathcal{Q}_{h}.
\label{mass1:SG}
\end{equation}
The construction of $a^{K}_{h}(\cdot,\cdot)$ and $m_h^{K}(\cdot,\cdot)$ satisfy the consistency (with respect to polynomials of $\mathbb{P}^{1}(K) $) and the stability properties. \answ{Moreover, we denote the matrix representation of the bilinear forms $a_h(\cdot,\cdot)$ and $m_h(\cdot,\cdot)$ by $\mathbf{A}$ and $\mathbf{M}$ respectively. In particular, $(\mathbf{A})_{ij}=a_h(\eta_i,\eta_j)$ and $(\mathbf{M})_{ij}=m_h(\eta_i, \eta_j)$, where $\eta_i$ represents the basis function.}
\begin{remark}
\answ{In the light of the previous discussion, we highlight that $S^K_a(\cdot,\cdot)$ and $S^K_m(\cdot,\cdot)$ are symmetric bilinear forms on $\mathcal{Q}(K) \times \mathcal{Q}(K)$ that scale like $a^K(\cdot,\cdot)$ and $m^K(\cdot,\cdot)$, respectively. Furthermore, there exist positive constants $\alpha_1$, $\alpha_2$, $\beta_1$ and $\beta_2$ such that, the following inequalities hold:
\begin{equation*}
    \begin{split}
        \alpha_1 \ a^{K}(v,v) & \leq S^{K}_{a}(v,v)  \leq \alpha_2 \  a^{K}(v,v)  \quad \forall v \in \mathcal{Q}(K) \cap Ker(\mathcal{P}^{\nabla}_K)  \\
\beta_1\  m^{K}(w,w) & \leq S^{K}_{m}(w,w)  \leq \beta_2 \  m^{K}(w,w) \quad \forall w \in \mathcal{Q}(K) \cap Ker(\mathcal{P}^0_K),
    \end{split}
\end{equation*}
where $Ker(\mathcal{P}^{\nabla}_K)$ and $Ker(\mathcal{P}^{0}_K)$ denote the null spaces of $\mathcal{P}^{\nabla}_K$ and $\mathcal{P}^{0}_K$, respectively. There are several choices of the projection operators in the literature \cite{mascotto2018ill,da2017high,dassi2018exploring}. However, we follow the construction provided in \cite{beirao2014hitchhiker}.} 

\answ{Note that the consistency part of the discrete bilinear forms $a_h(\cdot,\cdot)$ and $m_h(\cdot,\cdot)$ can be directly evaluated using the numerical quadrature.  However, in this study, we use the projection operators' matrix representation to evaluate the discrete bilinear form, thus circumventing the numerical quadrature. The projection operators' matrix elements can be assessed using the \answ{DOFs}\cite{beirao2014hitchhiker}.} 
\end{remark}
\textbf{Polynomial-consistency:} For all $h>0, ~ \forall\  K \in \mathcal{T}_{h}$, the bilinear form $a^{K}_{h}(\cdot,\cdot)$ and $m_h^K(\cdot,\cdot)$ \answ{satisfies} the following consistency property:
\begin{equation}
\begin{split}
a^{K}_{h}(q_{1},v_{h})&=a^{K}(q_{1},v_{h}), \\
m^K_h(q_{1},v_{h})&=(q_{1},v_{h})_K,
\label{poly_constncy}
\end{split}
\end{equation}
$\forall \  q_{1} \in \mathbb{P}^{1}(K)$ and  $v_{h} \in \mathcal{Q}(K)$. Further,  $a^K(u,v):=\int_K \nabla u \cdot \nabla v$ and $(u,v)_{K}:=\int_K u \ v$. \\

\textbf{Stability:} There exist two positive constants $\alpha_{\ast},\alpha^{\ast}$, independent of $h$ and $K$, such that, for all $v_{h}\in \mathcal{Q}(K)$, the following condition holds:
\begin{equation}
\begin{split}
\alpha_{\ast}\  a^{K}(v_{h},v_{h}) & \leq a^{K}_{h}(v_{h},v_{h})  \leq \alpha^{\ast} \  a^{K}(v_{h},v_{h})\\
\beta_{\ast}\  m^{K}(v_{h},v_{h}) & \leq m^{K}_{h}(v_{h},v_{h})  \leq \beta^{\ast}\  m^{K}(v_{h},v_{h}).
\label{stab1}
\end{split}
\end{equation}
The mass matrix associated with the bilinear form $m_{h}^{K}(\cdot,\cdot)$ assists to evaluate the nonlinear load term. Furthermore, the interpolation operator $I_h^{K}:H^{2}(K)\rightarrow \mathcal{Q}(K)$ and the polynomial approximation $m_{\pi}\in \mathbb{P}^1(K)$ of a function $m \in H^2(K)$ are introduced for each $K \in \mathcal{T}_h$. We briefly discuss the approximation properties of both the operators. For more details, interested readers are referred to~\cite{brenner2007mathematical,beirao2013basic}.
\begin{lemma}
\label{interpolation:SG}
For every $h>0,~K \in \mathcal{T}_h,~v \in H^2(K)$, the interpolant $I_h^{K} v \in \mathcal{Q}(K)$ satisfies
\begin{equation}
\label{interpol_1}
\|v-I_h^K v\|_{0,K}+h_K~|v-I_h^K~v|_{1,K} \leq~C h_K^2~|v|_{2,K},
\end{equation}
where $C$ is independent of the local mesh size $h_K$ but depends on the mesh regularity constant $\rho$.
\end{lemma}
\begin{proof}
\answ{See~Proposition~4.2 in \cite{mora2015virtual} for more detail}.
\end{proof}
We define the global interpolation operator $I_{h}:
I_{h}(v_{h})|_{K}:=I_{h}^{K}(v_{h}|_{K}).$
In addition, according to the classical Scott-Dupont principle \answ{\cite{beirao2016virtual,dupont1980polynomial}} and exploiting Assumption~\ref{assum_1}, there is a polynomial approximation in $\mathbb{P}_{1}(K)$, for all star-shaped polygonal element $K$ that satisfies the following approximation property.
\begin{lemma}
\label{poly_operator:SG}
For every $h>0,~K \in \mathcal{T}_h,~m \in H^2(K)$, there exists a polynomial $m_{\pi} \in \mathbb{P}_1(K)$ such that
\begin{eqnarray}
\|m-m_{\pi} \|_{0,K} +h_{K} \ |m-m_{\pi}|_{1,K} \leq C h_{K}^2 |m|_{2,K},
\label{intpol_err2}
\end{eqnarray}
where the positive constant $C$ is independent of the mesh size $h_K$ and could be a function of mesh regularity constant $\rho$.
\end{lemma}
\begin{proof}
 See Proposition~4.2 in \cite{beirao2013basic} for more detail.
\end{proof}
\begin{remark}
The bilinear forms $a_h^{K}(\cdot,\cdot)$ and $m_h^{K}(\cdot,\cdot)$ are symmetric positive (semi) definite on $\mathcal{Q}(K) \times \mathcal{Q}(K)$ and can be computed from the \answ{DOFs}. These bilinear forms consist of a polynomial and a nonpolynomial parts, which reduce to the analogous stiffness and mass matrix of the FEM for polynomial functions. A detailed discussion on the construction of stiffness and mass matrix can be found in \cite{cangiani2016conforming,beirao2014hitchhiker}.   
\end{remark}

\subsection{Construction of the nonlinear source term}
\label{comput_nonlinear}
In this section, we present a method to discretize the nonlinear load term exploiting the product approximation technique. Making use of analogous ideas as in~\cite{argyris1991finite}, we approximate the forcing vector (nonlinear load term). Given $\boldsymbol{\eta}=\{ \eta_1, \eta_2,\dots, \eta_{N_K}\}$, the local basis functions and $\mathcal{L}:= \{ l_1,l_2, \dots, l_{N_K}\}$, the \answ{DOFs} associated with the VEM space $\mathcal{Q}(K)$ over a polygon $K$, the product approximation for the forcing term is defined as 
\begin{equation}
\begin{split}
    f_h(u_h)|_{K}:=\mathcal{P}^{0}_{K}(I_h^{K}(-\sin(u_h))). \label{tic1}
\end{split}
\end{equation}
Employing approximation \eqref{tic1} and the orthogonality property of $\mathcal{P}^0_K$, the load term reduces to the following simple matrix structure:
\begin{equation}
\begin{split}
    F_h(\eta_i):=(f_h(u_h),\eta_i) =\int_\Omega \mathcal{P}^0I_h(f(u_h))~\eta_i & =\int_\Omega I_h(-\sin(u_h))~\mathcal{P}^{0}\eta_i \\
    & =\sum_{K \in \mathcal{T}_h} \int_K \sum_{j=1}^{N_K} l_j(-\sin(u_h))\ \eta_j\ \mathcal{P}^{0}_{K}\eta_i \\
    & = -\sum_{K \in \mathcal{T}_h} \sum_{j=1}^{N_K} \sin(u_h(V_j)) \int_K \mathcal{P}^{0}_{K}\eta_j\ \mathcal{P}^{0}_{K}\eta_i,
    \label{pa:SG}
    \end{split}
\end{equation}
where $V_j$ denotes the $j^{\rm th}$ vertex of a polygon $K$. Now, by varying the basis function $\eta_i$, $1 \leq i \leq N^{dof} $, we define the discrete forcing vector as:
\begin{equation}
    \tilde{f}:=\bar{\mathbf{M}}\ \bar{\mathcal{L}}(-\sin(u_h)),
    \label{nonlin_1}
\end{equation}
where $N^{dof}$ and $\bar{\mathcal{L}}$ denotes the dimension of the VEM space $\mathcal{Q}_h$ and the global \answ{DOFs}, respectively  and $(\bar{\mathbf{M}})_{ij}$=$\int_\Omega \mathcal{P}^{0} \eta_i \  \mathcal{P}^{0} \eta_j $. Further, $\bar{\mathcal{L}}(-\sin(u_h))$ denotes the column vector containing the values of $-\sin(u_h)$ at all the vertices.

\begin{remark}
From the definition of $f_h(u_h)$, it can easily be concluded that the linear functional $F_h(v_h):=(f_h(u_h),v_h)$ is computable from the \answ{DOFs} associated with the virtual element space $\mathcal{Q}_h$. Since the \answ{DOFs} of the linear virtual element space are evaluation of functions at the vertices, $l_j(\sin(u_h))=\sin(u_h(V_j))$ and therefore the unknowns $u_h(V_j)$ can be computed by standard iteration techniques such as the fixed point iteration technique or Newton's method with a suitable initial guess. However, this formulation may not  work for high order virtual element space, since the high order virtual element space ($k\geq2$) deals with cell moments.
\end{remark}

\subsection{Discrete scheme}
We employ the linear VEM for the discretization of the space variable and the Crank-Nicolson scheme for the temporal variable. Using  \eqref{temp6} and \eqref{mass1:SG}, the semidiscrete formulation of model problem \eqref{cont_weak:SG} reads as: find $u_h \in \mathcal{C}^0(0,T;\mathcal{Q}_h) \cap \mathcal{C}^1(0,T;\mathcal{Q}_h)$ such that~\cite{thomee1984galerkin,vacca2016virtual}:
\begin{equation}
\label{dis_modl:SG}
    \left \{
    \begin{aligned}
    & m_h(D^2_t u_h,v_h)+ \gamma \  m_h(D_t u_h,v_h)+a_h(u_h,v_h)=F_h(v_h) \  \forall v_h \in \mathcal{Q}_h, \text{for a.e.} \ 
     t \in (0,T],  \\
    & u_h(0)=h_0(x,y), \\
    & D_t u_h(0)=g_0(x,y),
    \end{aligned}
    \right.
\end{equation}
where $h_0(x,y)$ and $g_0(x,y)$ denote the interpolant of $h(x,y)$ and $g(x,y)$ on the virtual element space $\mathcal{Q}_h$. The existence and the uniqueness of the discrete solution $u_h$ follows from the fact that equations \eqref{dis_modl:SG} are equivalent to a nonlinear initial value problem of second order and $\sin(u)$ is globally Lipschitz continuous in $u$. Moreover, $F_h$ represents the discrete load term defined in \eqref{pa:SG}. Let $N \in \mathbb{N}$ be a positive integer and consider the time step $\Delta t:=T/N$ and for each $t_n=n \Delta t$, we define $u_h^n:=u_h(\cdot,t_n)$, where  $n=0,1,\ldots,N$. By employing the $\theta$-weighted scheme for the time discretization and the linear VEM for the spatial discretization, the fully discrete scheme is given by: find $u_h^{n+2} \in \mathcal{Q}_h$ such that
\begin{equation}
\left \{
    \begin{aligned}
        & m_h \Big(\frac{u_h^{n+2}-2\ u_h^{n+1}+u_h^n}{(\Delta t)^2},v_h \Big)  +  \gamma\  m_h\Big (\frac{u_h^{n+2}-u_h^{n}}{ 2 \ \Delta t},v_h \Big)+ \theta \  a_h(u_h^{n+2},v_h) \\
        & \quad +(1-\theta)\  a_h(u_h^{n},v_h) 
         = \Big( \ \theta F_h^{n+2}(v_h)+(1-\theta)F_h^{n}(v_h) \Big ) \ \forall \ v_h \in \mathcal{Q}_h,\\
         & u_h(0)=h_0(x,y), \quad \quad D_tu_h(0)=g_0(x,y). 
    \end{aligned}
    \right .
    \label{full_dis_1}
\end{equation}
where $F_h^{n}(v_h)=(f_h(u_h^{n}),v_h)$. In particular, for $\theta=$ 1/2, scheme \eqref{full_dis_1} reduces to the Crank-Nicolson scheme, which is unconditionally stable~\cite{li2008computational}. Using \eqref{nonlin_1}, Equation \eqref{full_dis_1} reduces to, in matrix form:
\begin{equation}
    \begin{split}
        & \bigg[ \Big(1+\frac{\gamma\ \Delta t}{2} \Big)\ \textbf{M}+(\Delta t^2/2)\  \textbf{A} \bigg]\ \mathbf{u}^{n+2}+ (\Delta t^2/2)   ~\answ{\bar{\mathbf{M}}}\ \sin(\mathbf{u}^{n+2}) =  2 \textbf{M}~\mathbf{u}^{n+1} \\
    & \quad -(\Delta t^2/2)~\answ{\bar{\mathbf{M}}}\ \sin(\mathbf{u}^{n})-\Big[ \Big (1- \frac{\gamma \Delta t}{2} \Big) \textbf{M} +(\Delta t^2/2) ~\textbf{A} \Big ]  \mathbf{u}^{n}.
    \label{cn2}
    \end{split}
\end{equation}
Equation \eqref{cn2} represents a system of nonlinear equations and we employ the Newton's method to solve the above system of equations. Since the computation of the matrices, $\textbf{M},~\bar{\mathbf{M}}$, and $\textbf{A}$ are done on the VEM space $\mathcal{Q}_h$ using the \answ{DOFs}, we can readily calculate \eqref{cn2} for any time step $t_n$. The residual vector is given by:
\begin{equation}
    \begin{split}
        & \boldsymbol{\mathcal{R}}(\bold{u^{n+2}}):=\bigg[ \Big(1+\frac{\gamma\ \Delta t}{2}\Big)\ \textbf{M}+(\Delta t^2/2)\  \textbf{A} \bigg]\ \mathbf{u}^{n+2}+ (\Delta t^2/2)\ \answ{\bar{\mathbf{M}}}\ \sin(\mathbf{u}^{n+2}) \\
        & \quad- 2 \textbf{M} \mathbf{u}^{n+1}+ (\Delta t^2/2)~\answ{\bar{\mathbf{M}}}\ \sin(\mathbf{u}^{n})+\Big[ \Big(1- \frac{\gamma \Delta t}{2}  \Big)\textbf{M}+(\Delta t^2/2) \textbf{A}\Big ]~\mathbf{u}^{n}.
    \end{split}
    \label{residual}
\end{equation}
and the Jacobian matrix corresponding to \eqref{cn2} is given by: 
\begin{equation}
\begin{split}
    \mathcal{J}_{ij}:=\frac{\partial \ \boldsymbol{\mathcal{R}}(\bold{u^{n+2}})_{i}  }{\partial \bold{u^{n+2}}_j}=\bigg[ \Big(1+\frac{\gamma\ \Delta t}{2} \Big )m_{ij}+(\Delta t^2/2) ~a_{ij} \bigg]+(\Delta t^2/2)~\bar{m_{ij}} \cos(\bold{u^{n+2}}_j),
    \end{split}
    \label{jacob}
\end{equation}
where, $m_{ij}$, $a_{ij}$ and $\bar{m_{ij}}$ are the elements of $\textbf{M}$, $\textbf{A}$ and $\bar{\mathbf{M}}$, respectively. We deduce from \eqref{residual} and \eqref{jacob} that the residual vector and the Jacobian are numerical \answ{integration free} and can be considered as readily computable matrix product. By employing the initial  condition  \eqref{cont_weak:SG}, $u^1$ is approximated as:
\begin{equation*}
    \frac{u^{1}(x,y)-u^{0}(x,y)}{ \Delta t}=I_h(g(x,y)).
\end{equation*}

\subsection{Comparison with an existing technique}
Here, we briefly recollect the approximation proposed for the load term $F_h(v_h)$ in \cite{adak2019convergence}. Employing the orthogonality property of the local $L^2$ projection operator $\mathcal{P}^{0}_{K}$, the nonlinear load term can be computed as
\begin{equation}
\begin{split}
F_h(v_h) &:=\sum_{K \in \mathcal{T}_h} \int_{K}  f_h(u_h)~v_h \\
                   & = \sum_{K \in \mathcal{T}_h} \int_{K}  -\sin \Big(\sum_{i=1}^{N_K} l_i(u_h)~ \mathcal{P}^{0}_{K} \eta_i \Big)~\mathcal{P}^{0}_{K} v_h,
\end{split}
\label{nonlin:semi}
\end{equation}
where $N_K$ denotes the total number of \answ{DOFs} of virtual space $\mathcal{Q}(K)$. Thus, if we replace the test function $v_h$ by $\eta_j$, we obtain the polynomial representation of the nonlinear load term, which is computable from the information provided by the \answ{DOFs} associated with the virtual element space $\mathcal{Q}(K)$. Further, the Jacobian matrix corresponding to approximation \eqref{nonlin:semi} can be computed as: 
\begin{equation}
\begin{split}
    & \mathcal{J}_{ij}=\Big[ \Big(1+\frac{\gamma\ \Delta t}{2}\Big)m_{ij}+(\Delta t^2/2) ~a_{ij} \Big] \\
    & \quad +(\Delta t^2/2)~\sum_{K\in \mathcal{T}_h} \int_K \cos \Big(\sum_{j=1}^{N_K} u^{n+2}_j~ \mathcal{P}^{0}_{K} \eta_j\Big)~\mathcal{P}^{0}_{K} \eta_j~\mathcal{P}^{0}_{K} \eta_i.
    \end{split}
    \label{jacob:2}
\end{equation}
It is seen that the above procedure involves integration over the element to compute the Jacobian, moreover, this has to be computed for each time step. Therefore, it can be inferred that the estimation of $\mathcal{J}_{ij}$ is computationally expensive when compared to that of \eqref{jacob}.

\section{Convergence Analysis:}
\label{converge:sg}
In this section, \answ{we} explore \answ{\textit{a priori}} error estimates for the semi-discrete and the fully discrete system in the $L^2$ norm and the $H^1$ seminorm. Since the force function $f(u)$ depends on $u$, the estimates will depend on the regularity of the discrete solution $u_h$ and $f(u_h)$.  In (Lemma~3 in \cite{larsson1989interpolation}, Lemma~2.1 and 2.3 in \cite{chen1989error}), the authors proved that $\|u_h\|_{L^2(0,T;H^2(\Omega))} < \infty$ and $\|f(u_h)\|_{L^2(0,T;H^2(\Omega))} < \infty$. However, to emphasize the estimate in terms of regularity of the exact solution $u$ and $f(u)$, we assume that  $\|u\|_{L^2(0,T;H^2(\Omega))} < \infty$ and $\|f(u)\|_{L^2(0,T;H^2(\Omega))} < \infty$. In connection with this assumptions, we introduce the elliptic projection operator $P^h:H^1_0(\Omega) \rightarrow \mathcal{Q}_h$, which is defined by
\begin{equation*}
    a_h(P^h u,v_h)=a(u,v_h) \quad \forall v_h \in \mathcal{Q}_h.
\end{equation*}
Now, we proceed to discuss the approximation properties of $P^h$, which will be utilized in the convergence analysis. Following Lemma~3.1 in \cite{vacca2016virtual}, we can derive the results of the following Lemma.
\begin{lemma}
\label{ell_pro:SG}
Let $u \in H^1_0(\Omega) \cap H^2(\Omega)$ and the domain $\Omega$ be convex. Then, there exists a generic constant $C$ independent of the mesh size $h$ such that
 \begin{equation}
    \|P^hu-u\|_0 \leq~C\ h^{2}~|u|_{2}, \quad \quad |P^h u-u|_1 \leq~C~h^{1}~|u|_{2}.
\end{equation}
\end{lemma}
Upon employing the projection operator $P^h$, we split the term $u-u_h$ as follows:
\begin{equation}
\begin{split}
    u-u_h &=u-P^h u+P^h u-u_h \\
          &=: \rho_1-\rho_2.
\end{split}
\end{equation}
The estimation of $\rho_1$ can be easily done by Lemma~\ref{ell_pro:SG}. To bound the term $\rho_2 (= u_h-P^h u)$, we employ the semi-discrete formulation~\eqref{dis_modl:SG} and the definition of the operator $P^h$. Note that most of the arguments are adopted from \cite{baker1980multistep} and incorporated in the VEM settings. We explicitly address below the boundedness of the nonlinear load term.

\subsection{Optimal $L^2$ error estimates}
\begin{theorem}
\label{th1}
Let $u \in L^2(0,T;H^2(\Omega))$ be the solution of \ (\ref{cont_weak:SG}) and $u_{h}$ be the solution of (\ref{dis_modl:SG}), and assume that the $f(u)\in L^{2}(0,T;H^{2}(\Omega))$, $D_tu \in L^{2}(0,T;H^{2}(\Omega))$, $D_t^2u \in L^{2}(0,T;H^{2}(\Omega))$, $u_0 \in H^{2}(\Omega)$ and $D_t u(0) \in H^{2}(\Omega)$. Additionally, let $u_h(0)=I_h(u_0)$ and $D_tu_{h}(0)=I_h(D_t u(0))$, where $I_h$ is the interpolation operator defined in \cite{beirao2014hitchhiker}. Then, $\forall \ t \in [0,T]$, the following estimate holds :
\begin{equation}
\begin{split}
\|(u-u_{h})(t)\| &\leq C~\Big(\|u_{h,0}-u_{0}\|+\|D_{t} (u-u_h)(0)\|\Big)+C~h^{2}~\Big(\|u\|_{L^{2}(0,T,H^{2}(\Omega))}
\\ & \quad +|u_0|_{2} +\|D_tu\|_{L^{2}(0,T,H^{2}(\Omega))} +\|D^2_{t}u\|_{L^{2}(0,T,H^{2}(\Omega))} \\
& \quad +\|f(u)\|_{L^{2}(0,T,H^{2}(\Omega))}\Big), 
\end{split}
\end{equation}
\answ{where $\|\cdot \|$ denotes $L^2(\Omega)$ norm}.
\end{theorem}

\begin{proof}
Using \eqref{cont_weak:SG} and \eqref{dis_modl:SG}, we can write
    \begin{equation}
\begin{split}
 m_{h}(D_{t}^{2}\rho_2(t), v_{h})+\gamma~m_h(D_t \rho_2(t), v_h) +a_{h}(\rho_2(t),v_{h})& =  (f_{h}(u_{h})-f(u),v_{h}) \\
 & \quad +(D^2_{t}u(t),v_{h})-m_{h}(P^hD_{t}^{2}u(t),v_{h}) \\
 & \quad + \gamma~(D_{t}u(t),v_{h})-\gamma~m_h(P^hD_{t}u(t),v_{h}).
\label{l2err_1:SG}
\end{split}
\end{equation}
\answ{Since the derivative in the time commutes with $m_h(\cdot,\cdot)$ and $(\cdot,\cdot)$, Equation~\eqref{l2err_1:SG} can be rewritten as follows:}
\begin{equation}
    \begin{split}
        -m_{h}(D_{t}\rho_2(t),D_{t}v_{h})+\gamma~m_h(D_t \rho_2(t), v_h )& +a_{h}(\rho_2(t),v_{h}) \\
        & = \answ{D_t} m_{h}(D_{t}(u-u_{h}),v_{h})+ \answ{D_t}(A_{1}(t),v_{h}) \\ & \quad-m_{h}(D_t^2 \rho_1,v_{h})
- m_{h}(D_{t}\rho_1,D_{t}v_{h})-(A_{1}(t),\answ{D_t}v_{h}) \\ & \quad +\answ{D_t}(A_{2}(t),v_{h}) -(A_{2}(t),\answ{D_t}v_{h}) \\ & \quad -\answ{D_t}m_{h}(A_{3}(t),v_{h})+m_{h}(A_{3}(t),\answ{D_t}v_{h})\\
& \quad + \gamma~\answ{D_t}(A_{4}(t),v_{h}) - \gamma \ (A_{4}(t),\answ{D_t}v_{h}) \\
& \quad -\gamma~\answ{D_t}m_h(A_{5}(t),v_{h}) + \gamma \ (A_{5}(t),\answ{D_t}v_{h}),
\label{l2err_2:SG}
    \end{split}
\end{equation}
where $A_1(t),A_2(t),A_3(t),A_4(t)$ and $A_5(t)$ are given by
\begin{equation}
\begin{split}
A_{1}(t)&:=\int^{t}_{0}(f_h(u_h)-f(u))(s)ds; \quad
 A_2(t):=\int^{t}_{0} D^2_{t}u(s)ds; \\
 A_3(t)&:=\int^t_{0} P^hD^2_{t}u(s)ds; \quad A_4(t):=\int^{t}_{0} D_{t}u(s)ds; \\
 A_5(t)&:=\int^t_{0} P^hD_{t}u(s)ds.
\label{l2err_3:SG}
\end{split}
\end{equation}
Define  $\hat{\rho_2}(t):=\int_{t}^{\xi} \rho_2(s) ds $, where  $\xi \in (0,T]$. Upon substituting $v_h=\hat{\rho_2}(t)$ into \eqref{l2err_2:SG}, and noting that $D_{t} v_h = -\rho_2(t)$, we obtain:
\begin{equation}
\begin{split}
 m_{h}(D_{t} \rho_2(t),\rho_2(t))+ &\gamma \answ{D_t} m_h(\rho_2(t),\hat{\rho_2}(t)) +\gamma~m_h(\rho_2(t),\rho_2(t))+a_{h}(\rho_2(t),\hat{\rho_2}(t)) \\
 &=\answ{D_t} m_h(D_t(u-u_h),\hat{\rho_2}(t))+\answ{D_t}(A_1(t),\hat{\rho_2}(t)) \\ & \quad -m_{h}(D^2_t\rho_1,\hat{\rho_2}(t))
 -m_h(D_t \rho_1,D_{t}\hat{\rho_2}(t))-(A_1(t),\answ{D_t} \hat{\rho_2}(t)) \\ & \quad +\answ{D_t}(A_2(t),\hat{\rho_2}(t)) -(A_2(t),\answ{D_t}\hat{\rho_2}(t)) \\ & \quad -\answ{D_t}m_h(A_3(t),\hat{\rho_2}(t))+m_h(A_3(t),\answ{D_t}\hat{\rho_2}(t))\\
 & \quad + \gamma \answ{D_t}(A_4(t),\hat{\rho_2}(t)) - \gamma (A_4(t),\answ{D_t}\hat{\rho_2}(t)) \\ & \quad - \gamma \answ{D_t}m_h(A_5(t),\hat{\rho_2}(t))+ \gamma \  m_h(A_5(t),\answ{D_t}\hat{\rho_2}(t)).
 \label{l2err_4:SG}
\end{split}
\end{equation}
We integrate \eqref{l2err_4:SG} with respect to $t$ from $0$ to $\xi$ and since $a_h(\hat{\rho_2}(0),\hat{\rho_2}(0)) \geq 0$ and \mbox{$m_h(\rho_2(t),\rho_2(t)) \geq 0$}, we deduce that:
\begin{equation}
\begin{split}
 \| \rho_2(\xi)\|^{2}_{h} &\leq \|\rho_2(0)\|^{2}_{h}-2 \gamma m_h(\rho_2(0),\hat{\rho_2}(0) )-2 m_{h}(D_{t}(u-u_{h})(0),\hat{\rho_2}(0)) -2\int_{0}^{\xi} m_{h}(D_{t}^2\rho_1,\hat{\rho_2}(t)) \\
 & \quad +2 \int_{0}^{\xi} m_{h} (D_{t}\rho_1,\rho_2(t)) +2 \int_{0}^{\xi}(A_1(t),\rho_2(t))+2\int^{\xi}_{0}(A_2(t),\rho_2(t)) \\
 & \quad -2\int^{\xi}_{0}m_h(A_3(t),\rho_2(t))+2\  \gamma \int^{\xi}_{0}(A_4(t),\rho_2(t))-2\  \gamma \int^{\xi}_{0}m_h(A_5(t),\rho_2(t)),
\label{l2err_5:SG}
\end{split}
\end{equation}
where $\|\cdot \|_h:=m_h(\cdot,\cdot)$ and the integrals in time of terms $A_1,A_2,A_3,A_4$ and $A_5$ are transformed through an integration by parts in time. An application of the continuity property of the discrete bilinear form $m_h(\cdot,\cdot)$, Young's inequality, the Cauchy-Schwarz inequality and the approximation property of $P^h$ (Lemma~\ref{ell_pro:SG}), yield the following estimates:
\begin{equation}
|m_{h}(D_{t}(u-u_{h})(0),\hat{\rho_2}(0)) | \leq C~ \|D_{t}(u-u_{h})(0)\|^{2}+~C~\answ{T^2}~\int_{0}^{\xi} \|\rho_2(t)\|^{2} dt ,
\label{l2err_6:SG}
\end{equation}
\begin{equation}
\begin{split}
 \Big| \int^{\xi}_0 m_{h}(D^2_t\rho_1(t),\hat{\rho_2}(t)) \Big | &\leq C~ \int_{0}^{\xi} \|D_{t}^2\rho_1(t)\|~ \|\hat{\rho_2}(t)\|\\
 & \leq~C~h^{4}~ \|D_{t}^2u \|^2_{L^{2}(0,T;H^{2}(\Omega))} +C~T^{2} \int_{0}^{\xi} \| \rho_2(t)\|^{2} dt,
 \label{l2err_7:SG}
 \end{split}
\end{equation}
and
\begin{equation}
\begin{split}
  \Big| \int^{\xi}_{0}m_h(D_{t}\rho_1,\rho_2(t)) \Big|  &\leq C~\int_{0}^{\xi} \| D_{t} \rho_1 \| ~\|\rho_2(t) \| \\
 &\leq C~ h^{4} \|D_{t}u\|^{2}_{L^{2}(0,T;H^{2}(\Omega))} + C~\int_{0}^{\xi} \| \rho_2(t) \|^{2} dt.
 \label{l2err_8:SG}
\end{split}
\end{equation}
Further, we split the load term as follows:
\begin{equation}
\begin{split}
    f_h(u_h)-f(u) &:=\mathcal{P}^0I_h(f(u_h))-f(u) \\
                 &= \mathcal{P}^0I_h(f(u_h))-\mathcal{P}^0I_h(f(u))+\mathcal{P}^0I_h(f(u))- \mathcal{P}^0f(u) \\
                 & \quad +\mathcal{P}^0f(u)-f(u).
                 \label{l2err_9:SG}
\end{split}
\end{equation}
Since $f(u)$ is a smooth enough, we can write:
\begin{equation}
    f(u_h(V_i))-f(u(V_i))=f'(\omega_i)(u_h(V_i)-u(V_i)) ,
\end{equation}
where, $|f'(\omega_j)| \leq 1$. Therefore, we infer that
\begin{equation}
\begin{split}
    I_h(f(u_h)-f(u))&=\sum_{j=1}^{N^{\text{dof}}} f(u_h(V_j))-f(u(V_j))~ \eta_j \\
    & =\sum_{j=1}^{N^{\text{dof}}}~f'(\omega_i)(u_h(V_i)-u(V_i))~\eta_j \\
    & \leq u_h-I_h(u).
    \label{l2err_10:SG}
\end{split}
\end{equation}
Upon employing the approximation property of $I_h$ and the boundedness of the $L^2$ orthogonal projection operator $\mathcal{P}^0$, we have
\begin{equation}
\begin{split}
    \|\mathcal{P}^0( I_h(f(u_h)-f(u))) \| &\leq~\|u_h-I_h(u) \| \\
                          & \leq~\|u_h-u\|+\|u-I_h(u) \| \\
                          & \leq~\|u_h-u\|+C~h^2~|u|_2.
                         \label{l2err_11:SG} 
\end{split}
\end{equation}
Further, since we assume $|f(u)|_2 < \infty$, and using the boundedness of the $L^2$ orthogonal projection operator and the approximation property of $I_h$, we obtain 
\begin{equation}
    \|\mathcal{P}^0I_h(f(u))- \mathcal{P}^0f(u)\| \leq~C~h^{2}~|f(u)|_2.
     \label{l2err_11_t1:SG} 
\end{equation}
and
\begin{equation}
    \|\mathcal{P}^0f(u)- f(u)\| \leq~C~h^{2}~|f(u)|_2.
    \label{l2err_11_t2:SG}
\end{equation}
An application of Lemma~\ref{interpolation:SG} and estimates~\eqref{l2err_11:SG}-\eqref{l2err_11_t2:SG}, yield the bound
\begin{equation}
    \|f_h(u_h)-f(u)\| \leq \|u-u_h\|+C~h^2~|u|_{2}+C~h^2~|f(u)|_2.
    \label{l2err_12:SG}
\end{equation}
Using analogous technique as in References~\cite{adak2018virtual,baker1980multistep}, we obtain:
\begin{equation}
\begin{split}
 \int^{\xi}_{0}(A_{1}(t),\rho_2(t)) &\leq C~h^{4}~\|u\|^2_{L^{2}(0,T,H^{2}(\Omega))}+C~h^{4}~\|f(u)\|^2_{L^{2}(0,T,H^{2}(\Omega))}+C~\int^{\xi}_{0}~\|\rho_2(t)\|^2 ~dt.
\label{l2err_13:SG}
\end{split}
\end{equation}
Using the polynomial consistency property of $m_h(\cdot,\cdot)$ (c.f. \ref{poly_constncy}) and the result of Lemma~\ref{poly_operator:SG}, we have,
\begin{equation}
\begin{split}
 \int^{\xi}_{0}\bigg((A_2(t),\rho_2(t))-m_h(A_3(t),\rho_2(t))\bigg)& \leq C~\Big(h^{4}~\|D^2_{t}u\|^2_{L^{2}(0,T,H^{2}(\Omega))} +\int^{\xi}_{0} \|\rho_2(t)\|^2 \Big),
\label{l2err_14:SG}
\end{split}
\end{equation}
and
\begin{equation}
\begin{split}
 \int^{\xi}_{0}\bigg((A_4(t),\rho_2(t))-m_h(A_5(t),\rho_2(t))\bigg)& \leq C~\Big(h^{4}~\|D_{t}u\|^2_{L^{2}(0,T,H^{2}(\Omega))} +\int^{\xi}_{0} \|\rho_2(t)\|^2 \Big).
\label{l2err_15:SG}
\end{split}
\end{equation}
substituting inequalities \eqref{l2err_6:SG}-\eqref{l2err_8:SG} and \eqref{l2err_13:SG}-\eqref{l2err_15:SG} into \eqref{l2err_5:SG} and using Grownwall's inequality and the stability of $m_h(\cdot,\cdot)$, we obtain:
\begin{equation}
\begin{split}
\|\rho_2(t)\|^2 &\leq \|\rho_2(0)\|^2+C~ \|D_{t}(u-u_{h})(0)\|^2+C~h^{4} \Big(\|D^2_{t}u\|^2_{L^2(0,T,H^{2}(\Omega))}\\
& \quad + \|D_{t}u\|^2_{L^2(0,T,H^{2}(\Omega))}+\|u\|^2_{L^2(0,T,H^{2}(\Omega))}+\|f(u)\|^2_{L^2(0,T,H^{2}(\Omega))}\Big).
\label{l2err_16:SG}
\end{split}
\end{equation}
Upon applying the approximation property of $P^h$, we write
\begin{equation}
\| \rho_2(0)\|^{2} \leq C~\Big( \|u_{h}(0)-u_{0}\|^{2}+h^{4} |u_{0}|_{2}^{2} \Big),
\label{l2err_17:SG}
\end{equation}
and
\begin{equation}
\|\rho_1(t)\|=\| (u-P^hu)(t) \| \leq C~ h^{2} \Big( |u_{0}|_{2}+|D_{t}u|_{L^{1}(0,T;H^{2}(\Omega))} \Big).
\label{l2err_18:SG}
\end{equation}
Using \eqref{l2err_16:SG}-\eqref{l2err_18:SG}, we prove the assertion of the Theorem: 
\begin{equation*}
\begin{split}
&\|(u-u_h)(t)\|  \leq \|\rho_1(t)\|+\|\rho_2(t)\| \\
& \leq C~ \Big(\|u_h(0)-u(0)\|+\|D_{t}(u-u_{h})(0)\| \Big) +C~h^{2}~\Big(|u_{0}|_{2}+\|u\|_{L^{2}(0,T;H^{2}(\Omega))}\\
&\quad +\|D_{t}u\|_{L^{2}(0,T;H^{2}(\Omega))} +\|D_{t}^2u\|_{L^{2}(0,T;H^{2}(\Omega))}+\|f(u)\|_{L^{2}(0,T;H^{2}(\Omega))}\Big).
\end{split}
\end{equation*}

\end{proof}
\begin{remark}
\answ{To derive the first term of the Equation~\eqref{l2err_2:SG} from the Equation \eqref{l2err_1:SG}, we have rewritten  Equation~\eqref{l2err_1:SG} into Equation~\eqref{l2err_2:SG} and adjusted with $A_i(t)$. Moreover, we would like to state that in order to derive $\|\rho(t)\|$, we have employed $D_t \hat{\rho}_2(t)$ as the test function. In addition, interested readers are referred to \cite{baker1980multistep} for detail discussion.}
\end{remark}

\subsection{Optimal $H^1$ error estimates}
\begin{theorem}
\label{th2}
Let $u \in L^2(0,T;H^2(\Omega))$ be the solution of \ (\ref{cont_weak:SG}) and $u_{h}$ be the discrete solution of (\ref{dis_modl:SG}). Moreover, assume that all the assumptions of Theorem~(\ref{th1}) holds. Then, there exists a positive constant $C(u,f(u))$ independent of the mesh size $h$, such that the following estimate holds:
\begin{equation}
\begin{split}
    |u(t)-u_{h}(t)|_{1} &\leq C~\Big(|u_{0}-u_{h,0}|_{1}+\|D_{t}(u-u_h)(0)\|\Big) +C~h^{1}\Big(|u_{0}|_{2} \\ & \quad + \|D_{t}u\|_{L^2(0,T,H^{2}(\Omega))} \Big)
+C~h^{2}~\Big(|D_{t}u(0)|_{2}+\|u\|_{L^2(0,T,H^{2}(\Omega))}   \\ & \quad +\|D^2_{t}u\|_{L^2(0,T,H^{2}(\Omega))}
+\|f(u)\|_{L^2(0,T,H^{2}(\Omega))} \Big) \quad \forall \ t \in [0,T].
\end{split}
\end{equation}
\end{theorem}
\begin{proof}
The convergence analysis of the error $u-u_h$ in the $H^1$ norm begins by considering $v_h=D_t\rho_2(t)$ in \eqref{l2err_1:SG}, which follows as:
\begin{equation}
\begin{split}
 m_{h}(D_{t}^{2}\rho_2(t), D_t \rho_2(t))+&\gamma~m_h(D_t \rho_2(t), D_t \rho_2(t)) +a_{h}(\rho_2(t),D_t \rho_2(t)) =  (f_{h}(u_{h})-f(u),D_t \rho_2(t)) \\
 & \quad +(D^2_{t}u(t),D_t \rho_2(t))-m_{h}(P^hD_{t}^{2}u(t),D_t \rho_2(t)) \\
 & \quad + \gamma~(D_{t}u(t),D_t \rho_2(t))-\gamma~m_h(P^hD_{t}u(t),D_t \rho_2(t)).
\label{h1err_1:SG}
\end{split}
\end{equation}
Since the time derivative commutes with $m_h(\cdot,\cdot)$, and $a_h(\cdot,\cdot)$, and $m_h(D_t \rho_2(t), D_t \rho_2(t))>0$, we deduce that:
\begin{equation}
\label{h1err:2}
\begin{split}
    & \frac{1}{2}~\answ{D_t}~m_h(D_t \rho_2(t),D_t \rho_2(t))+\frac{1}{2}~\answ{D_t}~a_h( \rho_2(t), \rho_2(t)) \\
    & \leq~|(f_{h}(u_{h})-f(u),D_t \rho_2(t))|+|(D^2_{t}u(t),D_t \rho_2(t))-m_{h}(P^hD_{t}^{2}u(t),D_t \rho_2(t))|\\
    &\quad + \gamma|(D_{t}u(t),D_t \rho_2(t))-m_h(P^hD_{t}u(t),D_t \rho_2(t))|.
    \end{split}
\end{equation}
Using analogous arguments as Theorem~4.3 in \cite{adak2018virtual}, we can bound the right hand side terms. 
An application of similar arguments as \eqref{l2err_12:SG}, we bound the nonlinear load term as
\begin{equation}
\label{tic:1}
    \begin{split}
        |(f_{h}(u_{h})-f(u),D_t \rho_2(t))| \leq~C \Big( \|u-u_h\|^2+h^4~|u|^2_2+h^4~|f(u)|_2^2 \Big)+ \|D_t \rho_2(t))\|^2.
    \end{split}
\end{equation}
Further, using the polynomial consistency property of $m_h(\cdot,\cdot)$ and Lemma~\ref{ell_pro:SG}, we can write the term as
\begin{equation}
\label{tic:2}
\begin{split}
|(D^2_{t}u(t),D_t \rho_2(t))-&m_{h}(P^hD_{t}^{2}u(t),D_t \rho_2(t))| \leq |(D^2_{t}u(t),D_t \rho_2(t))-(\mathcal{P}^0 D^2_{t}u(t),D_t \rho_2(t))| \\
& \quad +|(\mathcal{P}^0 D^2_{t}u(t),D_t \rho_2(t))-m_{h}(P^hD_{t}^{2}u(t),D_t \rho_2(t))|\\
& \leq \quad C~h^4~|D^2_{t}u(t)|^2_2+ \|D_t \rho_2(t)\|^2.
\end{split}
\end{equation}
Similarly, we deduce that
\begin{equation}
\label{tic:3}
    |(D_{t}u(t),D_t \rho_2(t))-m_h(P^hD_{t}u(t),D_t \rho_2(t))| \leq~C~h^4~|D_tu|_2^2+\|D_t \rho_2(t)\|^2.
\end{equation}
Upon substituting the estimations \eqref{tic:1} - \eqref{tic:3} into the Equation~\eqref{h1err:2} and an application of the stability properties of $m_h(\cdot,\cdot)$ and $a_h(\cdot,\cdot)$ (Equation \eqref{stab1}), the Cauchy Schwarz inequality, and integrating \eqref{h1err:2} from $0$ to $t$, we have
\begin{equation*}
    \begin{split}
        \|D_t\rho_2(t)\|^2+|\rho_2(t)|_1^2 &\leq ~C~(\|D_t\rho_2(0)\|^2+|\rho_2(0)|_1^2)+C~h^4 \Big(|u_{0}|_{2}^2+\|u\|_{L^{2}(0,T;H^{2}(\Omega))}^2\\
&\quad +\|D_{t}u\|_{L^{2}(0,T;H^{2}(\Omega))}^2 +\|D_{t}^2u\|_{L^{2}(0,T;H^{2}(\Omega))}^2+\|f(u)\|_{L^{2}(0,T;H^{2}(\Omega))}^2\Big) \\
& \quad + \int_0^t \|D_t\rho_2(s)\|^2 .
    \end{split}
\end{equation*}
Utilizing Gronwall's inequality and the estimation  $|\rho_1(t)|_1 \leq C~h \Big(|u_0|_2+|D_tu|_{L^1(0,T;H^2(\Omega))} \Big)$, we obtain the desired result.
\end{proof}
\begin{remark}
In the proof of Theorem~\ref{th2}, we skipped the proofs of the boundedness of the terms $|(f_{h}(u_{h})-f(u),D_t \rho_2(t))|$, and $|(D^2_{t}u(t),D_t \rho_2(t))-m_{h}(P^hD_{t}^{2}u(t),D_t \rho_2(t))|$, and $|(D_{t}u(t),D_t \rho_2(t))-m_h(P^hD_{t}u(t),D_t \rho_2(t))|$. These terms can be bounded easily using the polynomial consistency property of $m_h(\cdot,\cdot)$, \eqref{poly_constncy} and the polynomial approximation property (Lemma~\ref{poly_operator:SG}). We refer to Theorem~4.3 in \cite{adak2018virtual} for detailed evidence on the boundedness of the above mentioned terms.
\end{remark}

\subsection{Error estimation for the fully discrete scheme}
Now, we move to derive the error estimations for the fully discrete schemes. To present the analysis unambiguously, we introduce the following notations:
\begin{equation*}
    \partial_t^2 \rho_2^n:=\frac{\rho_2^{n+2}-2 \ \rho_2^{n+1}+\rho_2^n }{\Delta t^2}; \quad \delta_t \rho_2^n:=\frac{\rho_2^{n+2}-\rho_2^{n}}{2\ \Delta t}; \quad \rho_{2,1/2}^n:=\frac{\rho_2^{n+2}+\rho_2^n}{2};\quad \partial_t \rho_2:=\frac{\rho_2^1-\rho_2^0}{\Delta t},
\end{equation*}
where $\rho_2^n:=\rho_2(\cdot,t_n)$.

\begin{theorem}
Let $u$ be the analytical solution of \eqref{cont_weak:SG} and $\{ u_h^n\}_n$ be a sequence of discrete solutions satisfying \eqref{full_dis_1} \answ{for $\theta=1/2$}. Further, we assume that \answ{Assumption}~\ref{assum_1} holds and $\|\partial_t \rho_{2}\|+\|\rho_2^1\|_1+\|\rho_2^0\|_1=O(h^2+\Delta t^2)$, $D^4_{t}u \in L^2(0,T;L^2(\Omega))$, $D^3_{t} u \in L^{\infty}(0,T;L^2(\Omega))$. Additionally, let $f(u)\in L^{\infty}(0,T;H^{2}(\Omega))$, $D_tu \in L^{2}(0,T;H^{2}(\Omega))$, $D_t^2u \in L^{2}(0,T;H^{2}(\Omega))$. Then, there exists a positive constant $C$ such that the following estimation holds:
\begin{equation*}
\begin{split}
  \|u(t_n)-u_h^n\| &\leq ~C\Big(\Big \|\frac{\rho_2^1-\rho_2^0}{\Delta t} \Big \|+\|\rho_2^1\|_1+\|\rho_2^0\|_1  \Big) \\
  & \quad +C~ (\Delta t)^2~ \Big(\|D^4_tu\|_{L^2(0,T;L^2(\Omega))}+\|D^3_t u\|_{L^{\infty}(0,T;L^2(\Omega))} \Big) \\
    & \quad +C~h^2~ \Big(|u(0)|_2+|D_t u|_{L^{2}(0,T; H^2(\Omega))}+|D_t^2 u|_{L^{2}(0,T; H^2(\Omega))}+\|f(u)\|_{L^{\infty}(0,T; H^2(\Omega))} \Big). 
  \end{split}
\end{equation*}
\end{theorem}
\begin{proof}
Using the projection operator $P^hu$ at time $t=t_n$, we split the error $u(t_n)-u_h^n:=\rho_1^n-\rho_2^n$. Lemma~\ref{ell_pro:SG} yields the estimation of $\|\rho_1^n\|$. In order to prove the estimation for $\|\rho_2^n\|$, we proceed as follows: Using the equation~\eqref{full_dis_1}, we can write as
\begin{equation}
\label{full_discrete_pf:1}
\begin{split}
    m_h(\partial_t^2 \rho_2^n,v_h)& +\gamma~m_h(\delta_t \rho_2^n,v_h)+a_h(\rho_{2,1/2}^n,v_h)=(f_{h,1/2}^{n},v_h)-(f_{1/2}^n,v_h) \\
    & \quad- m_h(\partial_t^2 P^h u^n,v_h)+( D_t^2 u^n_{1/2},v_h)-\gamma~m_h(\delta_t P^h u^n,v_h)+\gamma~(D_t u^n_{1/2},v_h).
\end{split}
\end{equation}
Adding and subtracting the term $(\partial_t^2  u^n,v_h)$, we have
\begin{equation*}
    \begin{split}
        m_h(\partial_t^2 P^h u^n,v_h)&-(D_t^2 u^n_{1/2},v_h)=m_h(\partial_t^2 P^h u^n,v_h)-(\partial_t^2  u^n,v_h) \\
        & \quad + (\partial_t^2  u^n,v_h)-(D_t^2 u^n_{1/2},v_h).
    \end{split}
\end{equation*} 
Using analogous arguments as in \cite{dupont19732} (Lemma~6), we have 
\begin{equation*}
    |(\partial_t^2 u^n,v_h)-(D_t^2u^n_{1/2},v_h)| \leq C~\Delta t^3~\|D_t^4u\|_{L^2(t_{n-1},t_{n+1};L^2(\Omega))}~\|v_h\|.
\end{equation*}
Further, an application of the polynomial consistency property of $m_h(\cdot,\cdot)$ yields
\begin{equation*}
\begin{split}
    |m_h(\partial_t^2 P^h u^n, v_h)-(\partial_t^2 u^n,v_h)|&=|m_h(\partial_t^2 P^h u^n, v_h)-(\mathcal{P}^0 (\partial_t^2 u^n),v_h)|\\
    & \quad +|(\mathcal{P}^0 (\partial_t^2 u^n),v_h)-(\partial_t^2 u^n,v_h) |\\
    & \leq C~h^{2}~|\partial_t^2 u^n|_{2}~\answ{\|v_h\|}.
    \end{split}
\end{equation*}

Using the polynomial consistency property of $m_h(\cdot,\cdot)$, we have
\begin{equation*}
    \begin{split}
        |-m_h(\delta_t P^h u^n,v_h)+ (\delta_t  u^n, v_h)|& =|-m_h(\delta_t P^h u^n,v_h)+(\mathcal{P}^0(\delta_t u^n),v_h) |\\
        & \quad |-(\mathcal{P}^0(\delta_t u^n),v_h)+(\delta_t  u^n, v_h) |\\
        & \leq C~h^{2}~|\delta_t u^n|_{2}~\|v_h\|.
    \end{split}
\end{equation*}
An application of Taylor's theorem yields
\begin{equation*}
    |-(\delta_tu^n,v_h)+(D_t u^n_{1/2},v_h)| \leq~C ~(\Delta t)^2~\|D^3_t u(t_n)\|~\|v_h\|. 
\end{equation*}
Further, using the analogous arguments as \eqref{l2err_12:SG} at time $t=t_n$, we can bound the right hand side as: 
\begin{equation}
\label{load:discrete}
    \|f_h(u_h(t_n))-f(u(t_n))\| \leq~\|u(t_n)-u_h(t_n)\|+C~h^2~|u(t_n)|_2+C~h^2~|f(u(t_n))|_2.
\end{equation}
Substituting $v_h=\frac{\rho^{n+2}_2-\rho^{n}_2}{2\ \Delta t}$ in \eqref{full_discrete_pf:1}, and exploiting result \eqref{load:discrete} at time $t=t_n$, we obtain,
\begin{equation}
\label{full_dis:review}
\begin{split}
    & \Big( \Big\|\frac{\rho_2^{n+2}-\rho_2^{n+1}}{\Delta t } \Big 
    \|^2-\Big \|\frac{\rho_2^{n+1}-\rho_2^{n}}{\Delta t } \Big \|^2 +\gamma \Big \|\frac{\rho_2^{n+2}-\rho_2^{n}}{2 \ \Delta t } \Big \|^2+ 1/2~(|\rho^{n+2}_{2}|_1^2-|\rho^{n}_{2}|_1^2)\Big) \\
    & \leq C~(\Delta t)^4~\Big(\|D_t^4 u\|_{L^2(t_{n+2},t_{n};L^2(\Omega))}^2  +\|D_t^3u(t_n)\|^2 \Big)+C \Delta t~h^4~\Big(|u(t_n)|_2^2+|u(t_{n+2})|_2^2 \\
    & \quad + |f(u(t_n))|_2^2+|f(u(t_{n+2}))|_2^2+|\partial_t^2u^n|_2^2+|\delta_tu^n|_2^2\Big)+ \Big(\| \rho^{n+2}_2\|^2+\| \rho^{n}_2\|^2 \Big) \\
    & \quad + \Big(\Big \|\frac{\rho_2^{n+2}-\rho_2^{n+1}}{\Delta t } \Big \|^2+\Big \|\frac{\rho_2^{n+1}-\rho_2^{n}}{\Delta t }\Big \|^2 \Big).
\end{split}
\end{equation}
\answ{Following \cite{dupont19732}, we can express the term $\partial_t^2 u^n$ as 
\begin{equation*}
    \partial^2_t u_n=\frac{1}{\Delta t^2} \int_{-\Delta t}^{\Delta t} (\Delta t-|\tau|)~D_t^2u(t_{n+1}+\tau).
\end{equation*}
This implies
\begin{equation*}
    \Delta t~\sum_{n=0}^{n-2} |\partial^2_t u_n|_{2}^2 \leq~ \|D^2_t u\|_{L^2(0,T;H^{2}(\Omega))}.
\end{equation*}
Moreover, we can write the term as 
\begin{equation*}
    \begin{split}
         \Big (-m_h(\delta_t P^h u^n,v_h)+(D_t u^n_{1/2},v_h)\Big)&= \Big (-m_h(\delta_t P^h u^n,v_h)+ (\delta_t  u^n, v_h) \\
         &\quad -( \delta_t u^n, v_h)+(D_t u^n_{1/2},v_h)\Big).
    \end{split}
\end{equation*}}
Upon integration, Equation \eqref{full_dis:review} form $0$ to $n-2$ and using the discrete Gronwall's inequality, we have
\begin{equation}
    \begin{split}
        &\|\rho_2^n\|^2 \leq C \Big( \Big \|\frac{\rho_2^1-\rho_2^0}{\Delta t}\Big \|^2+\|\rho_2^1\|_1^2+\|\rho_2^0\|_1^2 \Big ) +C~(\Delta t)^4~\Big(\|D_t^4 u\|_{L^2(0,T;L^2(\Omega))}^2 \\& \quad +\|D_t^3u\|_{L^{\infty}(0,T;L^2(\Omega))}^2 \Big)  +C~h^4~\Big( |u(0)|_{2}^2+|D_t u|_{L^{2}(0,T; H^2(\Omega))}^2+|D_t^2 u|_{L^{2}(0,T; H^2(\Omega))}^2 \\
        & \quad +\|f(u)\|_{L^{\infty}(0,T; H^2(\Omega))}^2\Big).
    \end{split}
\end{equation}
Using the estimates of $\rho_1^n$ and $\rho_2^n$, we derive the final result:
\begin{equation*}
  \begin{split}
    & \|u(t_n)-u_h^n\|  \leq (\|\rho_1^n\|+\|\rho_2^n\|) \\
    &\leq ~C\Big(\Big \|\frac{\rho_2^1-\rho_2^0}{\Delta t}\Big \|+\|\rho_2^1\|_1+\|\rho_2^0\|_1  \Big)+C~ (\Delta t)^2~ \Big(\|D^4_tu\|_{L^2(0,T;L^2(\Omega))}+\|D^3_t u\|_{L^{\infty}(0,T;L^2(\Omega))} \Big) \\
    & \quad +C~h^2~ \Big(|u(0)|_2+|D_t u|_{L^{2}(0,T; H^2(\Omega))}+|D_t^2 u|_{L^{2}(0,T; H^2(\Omega))}+\|f(u)\|_{L^{\infty}(0,T; H^2(\Omega))} \Big).
   \end{split}
\end{equation*}
\end{proof}



\section{Numerical Experiments}
\label{num_ex_sg}
Four numerical examples are presented in this section to demonstrate the accuracy and the convergence characteristics of the proposed scheme. In all the cases, a rectangular domain is considered and the domain is discretized with different types of meshes, viz., distorted square, non-convex polygons and Voronoi meshes. The random polygonal mesh is generated using Polymesher~\cite{talischi2012polymesher,cangiani2016conforming}, a polygonal meshing tool based on Matlab. In addition, \answ{we} also consider concave elements and distorted quadrilateral elements. \fref{fig:mesh} shows a representative mesh for the distorted square, concave polygons and Voronoi meshes.  A pseudocode of the nonlinear solver is given in the Appendix.

\begin{figure}[htpb]
\centering     
\subfigure[Distorted Square]{\includegraphics[width=0.32\textwidth]{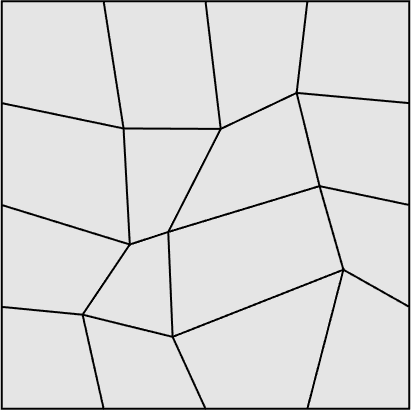}}  
\subfigure[Nonconvex]{\includegraphics[width=0.32\textwidth]{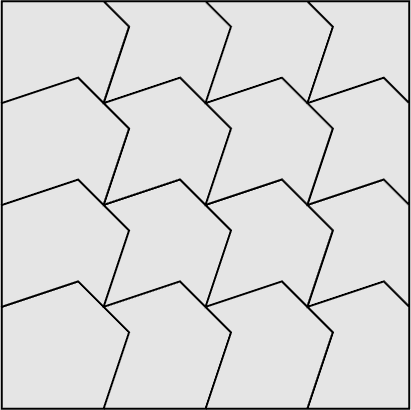}} 
\subfigure[Smoothed Voronoi]{\includegraphics[width=0.32\textwidth]{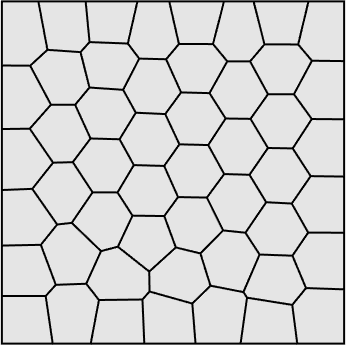}} 
\caption{A schematic representation of different discretizations employed in this study.}
\label{fig:mesh}
\end{figure}

\subsection{Test Problem-1} 
In the first example, we consider the following two-dimensional sine-Gordon equation
\begin{equation}
\left \{ 
\begin{aligned}
    & D^2_t u-\Delta u=f(u) \ \text{in} -7 \leq x,y \leq 7 \ \text{for} \ t \in (0,T] \\
    & u|_{\partial \Omega}=u_0,
\end{aligned}
\right.
\label{test_1}
\end{equation}
with the initial conditions:
\begin{equation*}
\begin{split}
    & u(x,y,0)=4\ \tan^{-1}(\exp (x+y)), \quad -7 \leq x,y \leq 7,   \\
    & D_t u(x,y,0)=-\frac{4\ \exp(x+y)}{1+\exp(2x+2y)}, \quad -7 \leq x,y \leq 7.
\end{split}
\end{equation*}
The analytical solution is given by:
\begin{equation}
    u(x,y,t)=4\ \tan^{-1}(\exp(x+y-t)) \quad t \in (0,T].
    \label{eqn:anasol}
\end{equation}
In \esref{test_1}, $u_0$ denotes the nonhomogeneous Dirichlet boundary condition that can be determined from the exact solution. The numerical solutions are computed over the nontrivial meshes in the rectangular domain \mbox{$-7 \leq x,y \leq 7 $} at time level $T=$ 1. The rate of convergence is evaluated through the relative error in the $L^2$ norm and the $H^1$ seminorm, defined by, which is the difference between the numerical solution $u_h$ and the interpolated solution $I_hu$ at final time $T$, as:
\begin{align}
    \|I_hu-u_h\|_0^2:= &\frac{m_h \Big(I_hu(\cdot,T)-u_h(\cdot,T),I_hu(\cdot,T)-u_h(\cdot,T)\Big)}{m_h \Big( I_hu(\cdot,T),I_hu(\cdot,T)\Big)} \\
    |I_hu-u_h|_1^2:=& \frac{a_h \Big(I_hu(\cdot,T)-u_h(\cdot,T),I_hu(\cdot,T)-u_h(\cdot,T)\Big)}{a_h \Big( I_hu(\cdot,T),I_hu(\cdot,T)\Big)}.
\end{align}

\fref{fig:l2h1_conv_example1} shows the convergence of the relative error in the $L^2$ and the $H^1$ norm with mesh refinement for a variety of meshes including smoothed Voronoi, distorted square and non-convex polygons for two different time steps, $\Delta t=$ 0.01, 0.5. For $\Delta t=$ 0.01, it can be inferred that the proposed method yields optimal convergence rate of 2 and 1 in $L^2$ and $H^1$, respectively as predicted theoretically in Theorems \ref{th1} and \ref{th2} (c.f. Section \ref{converge:sg}), when the mesh is refined. However, for $\Delta t=$ 0.05, the method yields optimal rate of convergence in the $H^1$ norm, but not in the $L^2$ norm. This can be attributed to the choice of $\Delta t$, which according to the time discretization scheme converges optimally only for sufficiently small $\Delta t$.

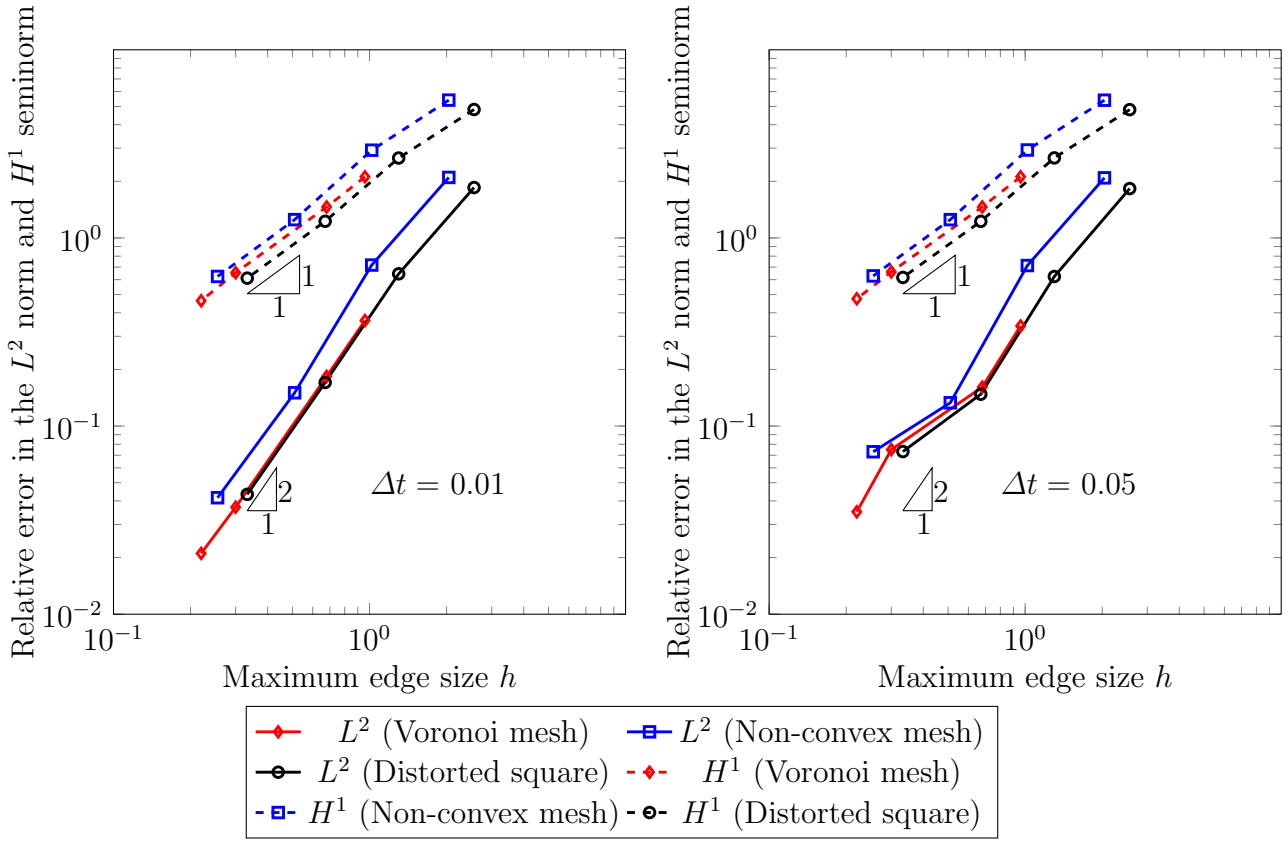
\begin{figure}[htpb!]
\newlength\figureheight 
\newlength\figurewidth 
\setlength\figureheight{10cm} 
\setlength\figurewidth{8cm}
%
%
\begin{center}
\begin{tikzpicture}
\begin{axis}[%
width=0.85092\figurewidth,
height=0.75092\figureheight,
scale only axis,
xmode=log,
xmin=0.1,
xmax=10,
xtick={0.1,   1},
xminorticks=true,
xlabel={Maximum edge size $h$},
ymode=log,
ymin=0.01,
ymax=10,
ytick={0.01,  0.1,    1},
yminorticks=true,
ylabel={Relative error in the $L^2$ norm and $H^1$ seminorm},
]
\addplot [color=red,solid,line width=1.1pt,mark=diamond,mark options={solid}]
  table[row sep=crcr]{%
0.96	0.36389\\
0.68	0.18356\\
0.3	0.037051\\
0.22	0.021048\\
};

\addplot [color=blue,solid,line width=1.1pt,mark=square,mark options={solid}]
  table[row sep=crcr]{%
2.04	2.0974\\
1.02	0.71763\\
0.51	0.15033\\
0.255	0.04161\\
};

\addplot [color=black,solid,line width=1.1pt,mark=o,mark options={solid}]
  table[row sep=crcr]{%
2.56	1.853\\
1.3	0.64468\\
0.671	0.17051\\
0.333	0.043356\\
};

\addplot [color=red,dashed,line width=1.1pt,mark=diamond,mark options={solid}]
  table[row sep=crcr]{%
0.96	2.1133\\
0.68	1.4608\\
0.3	0.65354\\
0.22	0.46349\\
};

\addplot [color=blue,dashed,line width=1.1pt,mark=square,mark options={solid}]
  table[row sep=crcr]{%
2.04	5.4026\\
1.02	2.9264\\
0.51	1.2512\\
0.255	0.62408\\
};

\addplot [color=black,dashed,line width=1.1pt,mark=o,mark options={solid}]
  table[row sep=crcr]{%
2.56	4.8149\\
1.3	2.6596\\
0.671	1.2259\\
0.333	0.61133\\
};

\addplot [color=black,solid,forget plot]
  table[row sep=crcr]{%
0.333	0.0355\\
0.433	0.0355\\
0.433	0.0600227209191173\\
0.333	0.0355\\
};
\node[right, align=left, inner sep=0mm, text=black]
at (axis cs:0.373,0.030175,0) {$1$};
\node[right, align=left, inner sep=0mm, text=black]
at (axis cs:0.435598,0.0447342647567873,0) {$2$};
\addplot [color=black,solid,forget plot]
  table[row sep=crcr]{%
0.333	0.505355\\
0.533	0.505355\\
0.533	0.808871516516517\\
0.333	0.505355\\
};
\node[right, align=left, inner sep=0mm, text=black]
at (axis cs:0.413,0.42955175,0) {$1$};
\node[right, align=left, inner sep=0mm, text=black]
at (axis cs:0.536198,0.629796771771772,0) {$1$};

\node[right, align=left, inner sep=0mm, text=black]
at (axis cs: 1.0,0.05,0) {$\Delta t=$ 0.01};
\end{axis}
\end{tikzpicture}%
~~~
\begin{tikzpicture}
\begin{axis}[%
width=0.85092\figurewidth,
height=0.75092\figureheight,
at={(0\figurewidth,0\figureheight)},
scale only axis,
xmode=log,
xmin=0.1,
xmax=10,
xtick={0.1,   1},
xminorticks=true,
xlabel={Maximum edge size $h$},
ymode=log,
ymin=0.01,
ymax=10,
ytick={0.01,  0.1,    1},
yminorticks=true,
ylabel={Relative error in the $L^2$ norm and $H^1$ seminorm},
legend style = {text=black},
legend to name=named,
legend columns = 2,
]
\addplot [color=red,solid,line width=1.1pt,mark=diamond,mark options={solid}]
  table[row sep=crcr]{%
0.96	0.34068\\
0.68	0.1611\\
0.3	0.074984\\
0.22	0.0350317\\
};
\addlegendentry{$L^2$ (Voronoi mesh)};

\addplot [color=blue,solid,line width=1.1pt,mark=square,mark options={solid}]
  table[row sep=crcr]{%
2.04	2.0863\\
1.02	0.71345\\
0.51	0.13299\\
0.255	0.073017\\
};
\addlegendentry{$L^2$ (Non-convex mesh)};

\addplot [color=black,solid,line width=1.1pt,mark=o,mark options={solid}]
  table[row sep=crcr]{%
2.56	1.8299\\
1.3	0.62373\\
0.671	0.14772\\
0.333	0.073208\\
};
\addlegendentry{$L^2$ (Distorted square)};

\addplot [color=red,dashed,line width=1.1pt,mark=diamond,mark options={solid}]
  table[row sep=crcr]{%
0.96	2.1151\\
0.68	1.4612\\
0.3	0.66002\\
0.22	0.47502\\
};
\addlegendentry{$H^1$ (Voronoi mesh)};

\addplot [color=blue,dashed,line width=1.1pt,mark=square,mark options={solid}]
  table[row sep=crcr]{%
2.04	5.3959\\
1.02	2.9341\\
0.51	1.2511\\
0.255	0.62827\\
};
\addlegendentry{$H^1$ (Non-convex mesh)};

\addplot [color=black,dashed,line width=1.1pt,mark=o,mark options={solid}]
  table[row sep=crcr]{%
2.56	4.8084\\
1.3	2.6629\\
0.671	1.2249\\
0.333	0.6166\\
};
\addlegendentry{$H^1$ (Distorted square)};

\addplot [color=black,solid,forget plot]
  table[row sep=crcr]{%
0.333	0.0355\\
0.433	0.0355\\
0.433	0.0600227209191173\\
0.333	0.0355\\
};
\node[right, align=left, inner sep=0mm, text=black]
at (axis cs:0.373,0.030175,0) {$1$};
\node[right, align=left, inner sep=0mm, text=black]
at (axis cs:0.435598,0.0447342647567873,0) {$2$};
\addplot [color=black,solid,forget plot]
  table[row sep=crcr]{%
0.333	0.505355\\
0.533	0.505355\\
0.533	0.808871516516517\\
0.333	0.505355\\
};
\node[right, align=left, inner sep=0mm, text=black]
at (axis cs:0.413,0.42955175,0) {$1$};
\node[right, align=left, inner sep=0mm, text=black]
at (axis cs:0.536198,0.629796771771772,0) {$1$};
\node[right, align=left, inner sep=0mm, text=black]
at (axis cs: 0.8,0.05,0) {$\Delta t=$ 0.05};
\end{axis}

\end{tikzpicture}%
\\

\ref{named}
\end{center} 
\caption{Example 1: convergence of the relative error in the $L^2$ and $H^1$ seminorm.}
\label{fig:l2h1_conv_example1}
\end{figure}

Next, we compare the performance of the proposed approach with that of conventional FE, when the domain is discretized with triangles. Table \ref{tabl:comptri_poly} presents the convergence of the relative error in the $L^2$ norm and $H^1$ seminorm. In this case, the domain is discretized with Voronoi meshes and structured triangular elements. The number of elements is kept the same in both the cases. It can be seen that the proposed method and the finite element method yield similar results. The main advantage of using polygonal mesh is that the meshing burden can be alleviated.
\begin{table}[ht]
\centering
 \caption{Convergence of the  error in the discrete $L^2$ norm and the $H^1$ norm: a comparison between the triangular and polygonal discretizations.}
 \label{tabl:comptri_poly}
\begin{tabular}{rcccccc}
\hline
\multicolumn{3}{c}{Polygons} && \multicolumn{3}{c}{Triangles} \\
\cline{1-3}\cline{5-7}
\answ{DOFs} & $L^2$ & $H^1$ && \answ{DOFs} & $L^2$ & $H^1$ \\
\hline
1002 & 7.9211$\times$10$^{-3}$ & 5.995$\times$10$^{-2}$ && 1502 & 2.1361$\times$10$^{-3}$ & 2.4044$\times$10$^{-2}$ \\
2002 & 2.6152$\times$10$^{-3}$ & 2.4342$\times$10$^{-2}$ && 3002 & 1.0277$\times$10$^{-3}$ & 1.1302$\times$10$^{-2}$\\
4002 & 1.0698$\times$10$^{-3}$ & 1.2770$\times$10$^{-2}$ && 6002 & 4.8927$\times$10$^{-4}$ & 6.9261$\times$10$^{-3}$ \\
10,002 & 3.3343$\times$10$^{-4}$ & 4.3689$\times$10$^{-3}$ && 15,002 & 2.7891$\times$10$^{-4}$ & 3.4745$\times$10$^{-3}$ \\
\hline
\end{tabular}
\end{table}

\subsection{Test Problem-2}
Next, we consider the following equation with a nonlinear source term:
\begin{equation}
D^{2}_{t}u-\Delta u= f(u)+g(x,y,t)\  \text{on} \quad \Omega \times I,
\label{num_ex_1}
\end{equation}
where $\Omega=[0,1]\times[0,1]$, and $f(u)=u^2$, and $I=(0,1]$ with $u=0$ on $\partial \Omega$. We choose $g(x,y,t)$ such that $u(x,y,t) =e^{-t}\ x\ y\ (1-x)\ (1-y)$ becomes the analytical solution of \eqref{num_ex_1}. In this example, we compare the performance of the proposed approach (product approximation technique) to treat the nonlinear term with that of the approach presented in~\cite{adak2019convergence}. The domain is discretized with Voronoi meshes and \answ{the} Crack-Nicolson scheme is employed for the temporal discretization. Table \ref{tbl:comp_fem} compares the relative error in the $L^2$ norm, CPU time in seconds and the number of Newton Iterations required between the present approach and the approach presented in~\cite{adak2019convergence}. The results are presented for the final time $T=$ 1 with a time step, $\Delta t=$ 0.01. It is inferred that the present approach requires fewer Newton iterations, which impacts the CPU time required for a similar order of accuracy. \answ{Further, we would like to highlight that the convergence analysis is performed for the nonlinear force function $f(u)=-\sin(u)$. Therefore this case ($f(u)=u^2$) is not covered by the theory. However, we have observed that the behaviour of the error is in accordance with that of the sine-Gordon nonlinearity and the error estimations for the case $f(u)=u^2$ can be done with standard modification. }

\begin{table}[ht]
\centering
 \caption{Convergence of the  error in the discrete $L^2$ norm, CPU time and number of Newton iterations: a comparison between the product approximate technique with the existing technique.}
 \label{tbl:comp_fem}
\begin{tabular}{rr|lll|lll}
\hline
           &       & Present method     &      &    & Ref.~\cite{adak2019convergence}         &      &    \\ 
           \cline{3-8}
Mesh size  & \answ{DOFs}   & $L^2$ Error & CPU-Time (s) & NI & $L^2$ Error & CPU-time (s) & NI \\
\hline
0.17468  & 162   & 3.2820$\times$10$^{-3}$   & 0.006974 & 1  & 3.2790$\times$10$^{-3}$   &0.118640      & 4  \\
0.08133 & 642  & 7.0280$\times$10$^{-4}$   & 0.040384 & 1  & 6.9781$\times$10$^{-4}$  & 0.533938 & 4  \\
0.04232 & 2601 & 1.5105$\times$10$^{-4}$   &0.381226  & 2  & 1.4099$\times$10$^{-4}$   & 2.4055 & 3  \\
0.02088  & 10601 & 3.2632$\times$10$^{-5}$   & 2.8086 & 2  & 3.1079$\times$10$^{-5}$   & 11.7883 & 3   \\ \hline
\end{tabular}
\end{table}

\subsection{Test Problem-3} 
To study the convergence behavior of the fully discrete scheme,  consider the problem~\eqref{num_ex_1} with $f(u)=-\sin u$ and $g(x,y,t)$ is chosen such that $u(x,y,t)=\sin(t) \sin(\pi x) \sin(\pi y)$ becomes the analytical solution. The results are reported in Table~\ref{table:time}. We emphasize that the optimal rate of convergence is observed in temporal direction for very small values of $h$. The computational domain, the mesh discretization, and final time are considered same as in the Test case~2. The analytical solution is smooth and profoundly depends on \answ{the} time variable.
\begin{table}[H]
\centering
\caption{Convergence of the relative error in the $L^2$ norm. The domain is discretized with \answ{smoothed Voronoi} element.}
\label{table:time}
 \begin{tabular}{cccc}
\hline
\multicolumn{1}{c}{} & \multicolumn{1}{c}{$\Delta t$} & \multicolumn{1}{c}{$L^2$ Error} & \multicolumn{1}{c}{Rate} \\ \hline
\multirow{4}{*}{$h=.010765$} & 1/5     & 4.8610$\times$10$^{-4}$   & -       \\
                       & 1/10    & 1.4540$\times$10$^{-4}$   & 1.74   \\
                       & 1/20    & 3.6105$\times$10$^{-5}$   & 2.00   \\
                       & 1/40    & 9.1523$\times$10$^{-6}$   & 1.98   \\
                                               
                        \hline \hline
\multirow{4}{*}{$h=.014894$} & 1/5     & 4.2010$\times$10$^{-4}$   & -        \\                                 
                       & 1/10    & 1.2403$\times$10$^{-4}$    & 1.76  \\
                       & 1/20    & 3.1441$\times$10$^{-5}$   & 1.98      \\
                       & 1/40    & 7.8060$\times$10$^{-6}$   & 2.01      \\\hline
                         
\end{tabular}
\end{table}
\subsection{Collision of four circular solitons}
As a last example, we consider the model problem (\ref{eqn:mod_prob}) with the dissipative coefficient $\gamma=$ 0.05 and the following initial conditions:
\begin{equation}
    \begin{split}
        & h(x,y)=4\ \tan^{-1}(\exp(\frac{4-\sqrt{(x+3)^2+(y+7)^2}}{0.436})), \quad -10 \leq x \leq 10, \quad -7 \leq y \leq 7 \\
        & g(x,y)=\frac{4.13}{ \cosh (\exp (\frac{4-\sqrt{(x+3)^2+(y+7)^2}}{0.436}))}, \quad \quad \quad \quad \quad \quad \  -10 \leq x \leq 10, \quad -7 \leq y \leq 7.
    \end{split}
\end{equation}
This is a problem of collision of four circular solitons. In this case, we consider homogeneous Neumann boundary condition. We represent the solution over the region $[-10,30]\times[-10,30]$ by extending across $x=$ -10 and $y=$ -10 and exploiting a symmetry relation. The numerical results are depicted in \fref{fig:four_ring_1} for a mesh size $h\approx$ 0.45, the time-step $\Delta t =$ 0.01. It is seen that at time $T=$ 0, the numerical solution yields four circular rings, which collide with each other and finally (at time $T=$ 11) yields one large ring. This is in agreement with the results presented in~\cite{dehghan2008numerical,argyris1991finite,sheng2005numerical}. 
\begin{figure}[htpb!]
\centering
\subfigure[$T=$ 0]{\includegraphics[width=.38\textwidth]{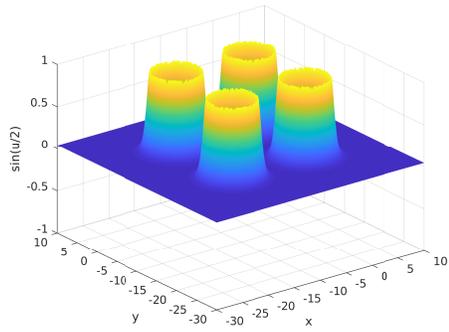}}
\subfigure[$T=$ 0]{\includegraphics[width=.38\textwidth]{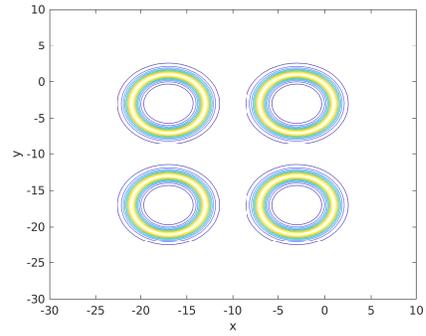}}
\subfigure[$T=$ 11]{\includegraphics[width=.38\textwidth]{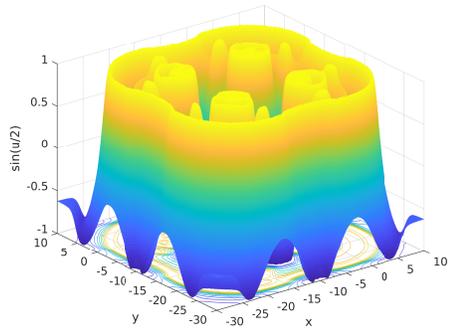}}
\subfigure[$T=$ 11]{\includegraphics[width=.38\textwidth]{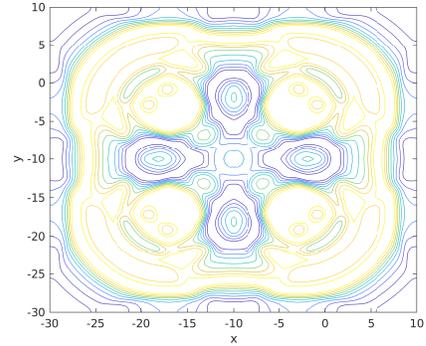}}
\caption{Collision of four circular ring solitons at time $T=$ 0 and $T=$ 11.}
\label{fig:four_ring_1}
\end{figure}

\section{Conclusion}
\label{con_sg}
In this paper, we employed the linear VEM to approximate the solutions of the two-dimensional sine-Gordon equation. The nonlinear source term is discretized by employing the product approximation technique. The salient feature of this approach is that the nonlinear term is `easily' computable form the \answ{DOFs} associated with the virtual element space. The proposed scheme offers a very simple Jacobian, which is numerically inexpensive to compute over general polygonal elements. \answ{Moreover, VEM is embedded with the advantage of mesh refinements (hanging nodes and interfaces are allowed) which makes computation cheaper}. The efficiency and the robustness of the proposed framework are demonstrated with numerical examples with domain discretized with general (possible concave) elements. 

\answ{\section{Extension of the product approximation technique \eqref{tic1} to the high order virtual element space}
In Section~\ref{comput_nonlinear}, we have discretized the nonlinear load term using the product approximation technique over the linear virtual element space. Within this, we can observe that the product approximation of the nonlinear force function is same as the interpolation of the force function over the virtual element space. Mathematically, this can be expressed as:
\begin{equation}
l_j(\sin(u_h))=\sin(u_h(V_j))=\sin(l_j(u_h)),
\label{extension:1}
\end{equation}
where $l_j$ corresponds to the evaluation of function $u_h$ at the vertex $V_j$ of a polygonal element $K$. The linear virtual element space deals only with the DOFs at the vertices\footnote{which are evaluation of the functions at the vertices}, which reduces the right hand side of the nonlinear system to a matrix structure, that can be easily evaluated by the mass matrix and the column vector consists of the unknowns (cf Equation~\eqref{nonlin_1}). 

As the DOFs for the high order virtual element space also includes the cell moments, the above reduction is not feasible. The cell moments are integral of the functions over the polygonal element, $K$~\cite{cangiani2016conforming}. In \eqref{extension:1}, if $l_j$ corresponds to the cell moment, i.e. \mbox{$l_j :=\frac{1}{|K|}~\int_K m_{\alpha} \ u_h$}, then we have:
\begin{equation*}
l_j(\sin(u_h))=\frac{1}{|K|}~\int_K \ m_{\alpha}\  \sin(u_h), \quad m_\alpha \in \mathcal{M}(K).
\end{equation*}
However, $\sin(l_j(u_h))=\sin \Big(\frac{1}{|K|}~\int_K \ m_{\alpha}\  u_h \Big)$, which imples $l_j(\sin(u_h))\neq \sin(l_j(u_h))$. An extension of the product approximation technique to the high order virtual element space, its theoretical estimation and numerical implementation will be a topic for future communications.}


\section*{Appendix}
\begin{algorithm}[H]
 Compute stiffness matrix $\mathbf{A}$ and mass matrix $\mathbf{M}$ \\
 \For {\rm K=1:nele}{
 1:~Evaluate $\mathbf{A}^K$ = $[a_h^{K}(\phi_i,\phi_j)]_{N^K \times N^K}$ \\
 2:~Evaluate $\textbf{M}^K$=$[m_h^{K}(\phi_i,\phi_j)]_{N^K \times N^K}$ \\
 3:~Evaluate $\bar{\mathbf{M}}=[m_h^{K}( \mathcal{P}^0_K\phi_i,\mathcal{P}^0_K\phi_j)]_{N^K \times N^K}$ \\
  4:~Assemble local stiffness and mass matrix to form global stiffness and mass matrix.
 }
  Solve the nonlinear system of equation using Crank-Nicolson method 
 \caption{Pseudo-code for the implementation of sine-Gordon equation}
  1:~Create a partition $0=t_{0}<t_{1}< \ldots <t_m=T$ of the time interval [0,T] with m time-steps $ \Delta t=t_l-t_{l-1}$. \\
 2:~Set $\mathbf{u}^{0}=I_h h(x,y)$\\
 3:~Compute $\mathbf{u}^{1}$ from initial condition\\
 4:~\For {\rm l=2,\ldots,m}{
 Solve nonlinear system of equation \eqref{full_dis_1}
 \begin{equation*}
 \begin{split}
     & \bigg[ \Big(1+\frac{\gamma\ \Delta t}{2} \Big)\ \textbf{M}+(\Delta t^2/2)\  \textbf{A} \bigg]\ \mathbf{u}^{l}- (\Delta t^2/2) \ \tilde{f}^{l} =  2 \  \textbf{M} \mathbf{u}^{l-1} \\
     & \quad +(\Delta t^2/2)~\tilde{f}^{l-1}-\Big[ \Big( 1- \frac{\gamma \  \Delta t}{2} \Big) \textbf{M}+(\Delta t^2/2)~ \textbf{A} \Big]~ \mathbf{u}^{l-2}.
     \end{split}
 \end{equation*}
 } 
 5:~Compute error using  $I_hu$ and numerical solution $\mathbf{u}^m$
\end{algorithm}


\newpage
\bibliographystyle{elsarticle-num}
\bibliography{reference_sg}

\end{document}